\documentclass[hidelinks,onefignum,onetabnum]{siamart251216}

\usepackage{lipsum}
\usepackage{amsfonts}
\usepackage{graphicx}
\usepackage{epstopdf}
\usepackage{algorithmic}
\usepackage{amssymb}
\ifpdf
  \DeclareGraphicsExtensions{.eps,.pdf,.png,.jpg}
\else
  \DeclareGraphicsExtensions{.eps}
\fi

\newsiamremark{remark}{Remark}
\newsiamremark{hypothesis}{Hypothesis}
\crefname{hypothesis}{Hypothesis}{Hypotheses}
\newsiamthm{claim}{Claim}
\newsiamremark{fact}{Fact}
\newsiamthm{ass}{Assumption}
\crefname{fact}{Fact}{Facts}

\headers{Complexity of an ITSQP Algorithm}{Michael J. O'Neill and Aoji Tang}

\title{Complexity of an inexact stochastic SQP algorithm for equality constrained optimization \thanks{\funding{This work was partially funded by the Office of Naval Research under award N00014-24-1-2638.}}}

\author {Michael J. O'Neill\thanks{Department of Statistics and Operations Research, The University of North Carolina at Chapel Hill, Chapel Hill, NC 
  (\email{mikeoneill@unc.edu}).}
\and Aoji Tang\thanks{Department of Statistics and Operations Research, The University of North Carolina at Chapel Hill, Chapel Hill, NC 
  (\email{ajtang@unc.edu}).}}

\usepackage{amsopn}

\newtheorem{thm}{Theorem}[section]
\newtheorem{lem}[thm]{Lemma}
\newtheorem{cor}[thm]{Corollary}

\newtheorem{ttt}{Termination Test}

\renewcommand{\leq}{\leqslant}

\newcommand{\eps}{\varepsilon}

\newcommand{\ie}{\emph{i.e.,}}

\newcommand{\uktrue}{u_k^\text{true}}
\newcommand{\dktrue}{d_k^\text{true}}
\newcommand{\yktrue}{y_k^\text{true}}
\newcommand{\wktrue}{w_k^\text{true}}
\newcommand{\ukt}{u_{k,t}}

\newcommand{\taum}{\tau_{\min}}

\let\ga=\alpha \let\gb=\beta   
             \let\go=\omega     
\let\gx=\chi

\newcommand{\cN}{\mathcal{N}}

\newcommand{\bE}{\mathbb{E}}

\newcommand{\bN}{\mathbb{N}}
\newcommand{\bR}{\mathbb{R}}
\newcommand{\bS}{\mathbb{S}}

\DeclareMathOperator{\argmin}{argmin}

\ifpdf
\hypersetup{
  pdftitle={Complexity of an inexact stochastic SQP algorithm for equality constrained optimization},
  pdfauthor={Michael J. O'Neill and Aoji Tang}
}
\fi

\usepackage{enumitem}
\usepackage{multirow}

\newcommand{\mo}[1]{#1}
\newcommand{\at}[1]{#1}

\begin{document}

\maketitle

\begin{abstract}
In this paper, we consider nonlinear optimization problems with a stochastic objective function and deterministic equality constraints. We propose an inexact two-stepsize stochastic sequential quadratic programming (SQP) algorithm and analyze its worst-case complexity under mild assumptions. The method utilizes a step decomposition strategy and handles stochastic gradient estimates by assigning different stepsizes to different components of the search direction. We establish the first known $\mathcal{O}(\epsilon_c^{-2})$ worst-case complexity with respect to the infeasibility measure when no constraint qualification is assumed and a worst-case complexity of $\mathcal{O}(\epsilon_c^{-1})$ when LICQ holds, matching the best known result in the literature. In addition, under mild conditions, our method achieves the optimal $\mathcal{O}(\epsilon_L^{-4})$ complexity with respect to the gradient of the Lagrangian regardless of constraint qualifications. Our results provide the first complexity guarantees for the popular Byrd-Omojukun step decomposition strategy and verify its theoretical efficacy. Numerical experiments show that our algorithm has a superior infeasibility convergence performance and a competitive KKT convergence rate compared to the state-of-the-art stochastic SQP method.
\end{abstract}



\section{Introduction}

In this paper, we consider the equality-constrained optimization problem as follows:
\begin{equation}\label{eq: opt problem}
 \min _{x\in \bR^n} f(x) = \bE[F(x, \omega)] \quad \text{s.t.} \quad c(x) = 0,
\end{equation}
where $f$ is the expectation of $F:\bR^n\times \Omega \to \bR$, under the random variable $\omega$ with the associated probability space $(\Omega, \mathcal{F,P})$, and $c:\bR^n\to \bR^m$ is the constraint function.

Such problems emerge naturally in scenarios where system dynamics are strictly governed by physical laws. For instance, in stochastic optimal power flow problems, the objective is to minimize expected generation costs under highly stochastic renewable energy inputs, while the power flow balance must strictly hold as deterministic equality constraints dictated by Kirchhoff's circuit laws. Similarly, in Physics-Informed Deep Learning (PINN) problems \cite{chen2024physics, lu2021physics}, boundary and initial conditions are often enforced as hard deterministic constraints, while \mo{only stochastic gradient information} is available \mo{for computational efficiency}. Such formulations are also prevalent in science and engineering, with prominent applications in optimal control \cite{betts2010practical, teo2021applied} and PDE-constrained optimization \cite{kouri2018inexact}.

There are two general types of methods to solve our desired problem. Stochastic augmented Lagrangian methods \cite{bollapragada2023adaptive, jiang2022stochastic, li2024stochastic, shi2025momentum, xu2020primal} penalize the constraints in the objective and solve the corresponding unconstrained subproblem to get a \mo{primal update}. \mo{These methods} can benefit from simpler forms of constraints \mo{(e.g. linear constraints)} and \mo{naturally} generalize to inequality constrained problems. Meanwhile, stochastic Sequential Quadratic Programming (SQP) methods are better at coping with more complicated nonlinear constraints. SQP methods compute the \mo{iterates} via a sequence of quadratic programming subproblems which preserve linearized \mo{feasibility}.

Existing stochastic SQP frameworks \cite{berahas2022adaptive, berahas2021sequential, berahas2023accelerating, berahas2025optimistic, berahas2025sequential, curtis2024stochastic, na2023adaptive, o2024two} rely heavily on the Linear Independence Constraint Qualification (LICQ) at every iteration. However, this assumption often fails in high-dimensional applications. For example, in PINN problems, enforcing boundary conditions over dense collocation points often leads to redundant constraints, inherently causing rank-deficient constraint Jacobians. When LICQ fails, the KKT system becomes singular. This singularity disrupts standard Newton-type step computations and can cause the merit parameter to vanish prematurely, preventing the algorithm from adequately optimizing the objective function.

To safeguard against these degenerate constraint manifolds, \mo{one popular approach is to employ the step decomposition strategy of Byrd-Omojokun} \cite{omojokun1989trust}. The Byrd-Omojokun method computes two orthogonal components: the normal component (oriented toward feasibility), \mo{via a trust-region subproblem}, and the tangential component (optimizing the objective within the null space), \mo{through the solution of a linear system of equations. These two components are then combined to form the full step direction. This strategy has proven effective for stochastic SQP methods in recent work \cite{berahas2024stochastic}.}

\subsection{Contributions}

Given the critical absence of complexity results \mo{under mild assumptions}, we propose and analyze an Inexact Two-Stepsize SQP (ITSQP) algorithm. By employing an $l_2$ merit function, we comprehensively analyze the worst-case complexity results under three specific behaviors of the merit parameter sequence. We prove that our algorithm converges to stationary points (in expectation) in both infeasibility and first-order measures. 


Our specific contributions \mo{in terms of} worst-case complexity \mo{results} are as follows:

\begin{itemize}
    \item \textbf{Without Constraint Qualifications:} To achieve an $\epsilon_c$-feasible point, our algorithm requires at most $\mathcal{O}(\epsilon_c^{-2})$ iterations \mo{regardless of whether the} merit parameter sequence has a positive lower bound. To our knowledge, this is the \emph{first known} complexity result for stochastic SQP algorithm under such mild, rank-deficient assumptions.
    \item \textbf{With LICQ:} We prove that the Byrd-Omojokun safeguard does not impede performance. When everywhere LICQ holds, the algorithm perfectly recovers an improved $\mathcal{O}(\epsilon_c^{-1})$ worst-case complexity for the infeasibility measure. This result matches the best known complexity in \cite{o2024two}, \mo{where LICQ is assumed}.
    \item \textbf{First-order Optimality:} When the merit parameter sequence is bounded, our algorithm requires at most $\mathcal{O}(\epsilon_{L}^{-4})$ iterations to ensure that the first-order optimality condition \mo{falls below the tolerance} $\epsilon_L$, which matches the \mo{optimal complexity bounds} for stochastic gradient methods \cite{arjevani2023lower}.
\end{itemize}

Furthermore, we exploit the use of inexact iterative solvers in the tangential subproblems and prove \mo{equivalent} convergence and complexity results under mild assumptions on these solvers. Numerical experiments demonstrate that our ITSQP algorithm yields substantial enhancements in both feasibility convergence speed and accuracy compared to state-of-the-art exact stochastic SQP methods.

We conclude this section with \cref{tab: complexities}, which contains the complexity results of different algorithms. Note that the \mo{approximate convergence metrics} vary under different assumptions.

\begin{table}[htbp]
    \centering
    \begin{tabular}{|c|c|c|c|c|}
    \hline
     Algorithm  &  Assumption & $1^\text{st}$-order Comp. & Infeas. Comp. \\
     \hline
         SPD\cite{jin2022stochastic} &  $x_0$ feasible, CQ & $\mathcal{O}(\epsilon_L^{-5})$ & $\mathcal{O}(\epsilon_c^{-5})^\dagger $\\
    \hline
        MLALM\cite{shi2025momentum} & mean-squared smoothness, CQ & {$\mathcal{O}(\epsilon_L^{-4})$} & {$\mathcal{O}(\epsilon_c^{-4})^\dagger$}
        \\
    \hline
        SSQP\cite{berahas2021sequential} &  LICQ &  $\mathcal{O}(\epsilon_L^{-4})$ & $\mathcal{O}(\epsilon_c^{-2})$ \\
    \hline 
        SQP-AL\cite{na2023adaptive} & LICQ & $\mathcal{O}(\epsilon_L^{-4})$ & $\mathcal{O}(\epsilon_c^{-4})$ \\
    \hline
        TSSQP\cite{o2024two} &  LICQ &  $\mathcal{O}(\epsilon_L^{-4})$ & $\mathcal{O}(\epsilon_c^{-1})$ \\
    \hline
        \multirow{3}{*}{\shortstack{ITSQP \\ \small (this paper)}}
         & \textbf{No CQ assumed}, $\exists \ \taum$
          & $\mathcal{O}(\epsilon_L^{-4})$ 
          & $\mathcal{O}(\epsilon_c^{-2})^\ddagger$ \\
        \cline{2-4}
        & \textbf{No CQ assumed}, $\nexists \ \taum$
          & N/A
          & $\mathcal{O}(\epsilon_c^{-2})^\ddagger$ \\
        \cline{2-4}
         & LICQ
          & $\mathcal{O}(\epsilon_L^{-4})$ 
          & {$ \mathcal{O}(\epsilon_c^{-1})$} \\
    \hline
    \end{tabular}
    \caption{Complexity results of different algorithms solving \cref{eq: opt problem}. The default first-order measure is the norm of gradient of the Lagrangian function. The default infeasibility measure is the norm of constraint violation. $\dagger$: The stochastic augmented Lagrangian methods use a combined feasibility and optimality measure. $\ddagger$: We use a different infeasibility measure when no CQ is assumed, \mo{see \cref{cor: Jkck2 complexity}. Here, $\taum$ represents a lower bound on the merit parameter seqeunce.}}
    \label{tab: complexities}
\end{table}

\subsection{Organization}
The assumptions and our algorithm are presented in \Cref{sec: problem statement and algorithm description}. \Cref{sec: convergence analysis} details the worst-case complexity proofs across the three \mo{aforementioned} cases. \Cref{sec: inexact algs} discusses the implementation of inexact subproblem solutions and proves complexity results \mo{for the inexact method}. \mo{Preliminary} experiments are presented in \Cref{sec: numerical experiments} and \mo{we provide a brief discussion of our work} in \Cref{sec: conclusion}.

\subsection{Notation}
The set of real numbers is denoted as $\bR$, the set of positive real numbers is denoted as $\bR_{>0}$, the set of natural numbers is denoted as $\bN$, the set of $n$-dimensional real vectors is denoted as $\bR^n$, the set of $m\text{-by-}n$ dimensional real matrices is denoted as $\bR^{m\times n}$, and the set of $n$-by-$n$ dimensional symmetric matrices is denoted as $\bS^n$. \mo{We use $\|\cdot\|$ to denote the Euclidean norm.} As in \cref{eq: opt problem}, $f$ and $c$ denote the objective and constraint functions, respectively. $c_i$ denotes the $i$-th component of constraint $c$. Given $A\in \bR^{m\times n}$, the null space of $A$ is denoted as $\text{Null} (A)$, and the range space of $A^T$ is denoted as $\text{Range} (A^T)$. \mo{We denote the Moore-Penrose psuedoinverse of a matrix $A$ as $A^+$.}

Our algorithm is iterative in the sense that, given a starting point $x_0\in \bR^n$, it generates a sequence of iterations $\{x_k\}_{k\ge0}$ with $x_k\in \bR^n$. At iterate $k$, we denote $c_k = c(x_k)$ and $J_k = \nabla c(x_k)^T$. Given $J_k\in \bR^{m\times n}$, $Z_k$ denotes a matrix whose columns form an orthonormal basis for $\text{Null} (J_k)$.

\section{Problem Statement and Algorithm Description}\label{sec: problem statement and algorithm description}

We make the following assumption about the problem \cref{eq: opt problem} throughout the paper. 

\begin{ass}\label{assumption1}
    Let $\gx \subset \bR^n$ be an open convex set containing the sequence of iterations $\{x_k\}$ generated by our algorithm. The objective function $f : \bR^n\to \bR$ is continuously differentiable and bounded over $\gx$ and its gradient $\nabla f: \bR^n \to \bR$ is Lipschitz continuous with constant $L > 0$ and bounded over $\gx$. The constraint function $c: \bR^n\to \bR^m$ is continuously differentiable and bounded over $\gx$ and its Jacobian function $J:= \nabla c^T : \bR^n \to \bR^{m\times n}$ is bounded over $\gx$ and Lipschitz continuous with constant $\Gamma > 0$ (with respect to $\|\cdot\|$) over $\gx$.
\end{ass}
    
From Assumption \ref{assumption1}, it follows that there exist positive real numbers $(f_{\inf}, $ $\kappa_g, \kappa_c, \kappa_J)$ $ \in \bR \times \bR_{>0} \times \bR_{>0} \times \bR_{>0}$, such that 
$$
f_{\inf} \le f(x_k) , \,\,\|\nabla f(x_k)\|\le \kappa_g, \,\, \|c_k\|\le \kappa_c, \,\,\text{and} \,\, \|J_k\| \le \kappa_J \,\, \text{for all} \,\, k\in \bN.
$$
Assumption \ref{assumption1} are common assumptions on the smoothness of the functions. As for the boundedness of \mo{the} function and gradient, it is not ideal to assume \mo{these remain bounded in} a stochastic setting, since we can only provide convergence in expectation. \mo{However, such assumptions are essential in {\it deterministic} constrained optimization and are common in the stochastic constrained literature \cite{berahas2024stochastic,berahas2021sequential,o2024two,shi2025momentum}}.

Now we define the Lagrangian function $L:\bR^n \times \bR^m \to \bR$ corresponding to the problem \cref{eq: opt problem} by $L(x, y) = f(x) + c(x)^Ty$. The standard KKT condition under the linear independence constraint qualification \mo{is} 
\begin{equation}\label{eq: KKT condition of Lagrangian}
    \begin{bmatrix}
    \nabla_xL(x, y) \\ \nabla_y L(x,y)
\end{bmatrix} = \begin{bmatrix}
    \nabla f(x) + J(x)^Ty \\ c(x)
\end{bmatrix} = 0.
\end{equation}
However, without any kind of constraint qualification, it is possible for the KKT system to \mo{be} degenerate due to the possible rank deficiency of the constraint Jacobian matrix $J(x)$. Then \cref{eq: KKT condition of Lagrangian} might not necessarily be satisfied at a solution to the original problem \cref{eq: opt problem}, or the problem \cref{eq: opt problem} might be infeasible itself. In the latter case, we can only hope that some other measure of constraint violation converges to zero.

To handle this, we employ the $l_2$-norm of constraints, denoted by $\varphi (x) = \|c(x)\|$, as our infeasibility measure. We say that a point $x\in \bR^n$ is stationary with respect to $\varphi$, if and only if either:
\begin{enumerate}
    \item $\varphi(x) = 0$, or
    \item $\varphi(x) \neq 0$ and $\nabla \varphi (x) = \frac{J(x)^Tc(x)}{\|c(x)\|} = 0$.
\end{enumerate}
We acknowledge that this approach is a compromise, but it is common practice when no constraint qualification is assumed, as in \cite{berahas2024stochastic}. In our analysis, we mainly \mo{work with} $\|J(x)^Tc(x)\|$ as \mo{our} infeasibility measure. This is because $\varphi(x) = 0 $ will also lead to $\|J(x)^Tc(x)\|=0$, and $\varphi$ can be recovered from $\|J(x)^Tc(x)\|$ when the LICQ \mo{holds at} every iteration. \mo{When considering first-order stationarity, we use the gradient of the Lagrangian, $\|\nabla f(x)+J(x)^Ty\|$, with a properly chosen dual variable $y$.}

\subsection{Step computation and algorithm}

\mo{Standard stochastic SQP algorithms compute} the search direction via a sequence of quadratic programming subproblems which \mo{satisfy linearized feasibility}. At iteration $k$, the subproblem is formed as
\begin{equation}\label{eq: stochastic quadratic subproblem}
    \begin{aligned}
    \mo{d_k} = \underset{d\in\bR^n}{\argmin}& \quad  f_k + g_k^Td +\frac{1}{2}d^TH_kd \\ \text{s.t.}& \quad c_k + J_k d = 0,
\end{aligned}
\end{equation}
where $H_k$ is \mo{a chosen matrix} and $g_k$ is a stochastic estimate of $\nabla f(x_k)$.

For a general quadratic programming problem in the form \cref{eq: stochastic quadratic subproblem}, there are numerous ways to \mo{compute the solution} when $J_k$ has full row rank, such as factorization methods and conjugate gradient methods \cite{nocedal2006numerical}. However, \mo{as we do not assume full rank of the Jacobians, rather than solving \eqref{eq: stochastic quadratic subproblem} directly for $d_k$, we \mo{employ} a step decomposition strategy known as the Byrd-Omojokun method \cite{omojokun1989trust}}. 

At each iteration, we first compute the ``normal" component $v_k\in \text{Range}(J_k^T)$ of the search direction to minimize linearized constraint violation over a trust-region: 
\begin{equation}\label{eq: subproblem v_k}
    \min_{v\in\bR^n} \frac{1}{2}\|c_k+J_kv\|^2 \quad \text{s.t.} \quad \|v\| \le \go \|J_k^Tc_k\|,
\end{equation}
where $\go >0$ is a deterministic parameter to control the size of the normal component. Rather than using an adaptive trust-region radius sequence $\{\Delta_k\}$ (like in \cite{fang2024fully}), we intentionally link $v_k$ with our infeasibility measure $\|J_k^Tc_k\|$. 


Solving \cref{eq: subproblem v_k} exactly may be expensive, but fortunately we only \mo{require that} the normal step $v_k$ \mo{satisfies} the following Cauchy decrease condition: 
\begin{equation}\label{eq: cauchy decreasee condition}
    \|c_k\| - \|c_k+J_kv_k\| \ge \epsilon_v (\|c_k\| - \|c_k+\alpha_k^C J_kv_k^C\|),
\end{equation} for some $\epsilon_v \in (0, 1]$. Here, $v_k^C = -J_k^Tc_k$, and $\alpha_k^C$ is the solution to the problem $\min_\alpha \frac{1}{2}\|c_k + \alpha J_kv_k^C\|^2$, subject to $\alpha \le \go$. Since this condition can be satisfied simply by choosing the Cauchy step $v_k \leftarrow v_k^C$, the normal component can be computed at a relatively low cost. To obtain a more accurate solution, one can apply the linear conjugate gradient method with Steihaug stopping conditions \cite{Steihaug1983} to find a solution that satisfies the Cauchy decrease condition. \mo{An important property of \cref{eq: subproblem v_k} is that the normal component $v_k$ is independent of the stochasticity introduced by $g_k$}. This property \mo{is fundamental to our} two-stepsize scheme and lays the theoretical foundation for our superior infeasibility complexity compared to that of \cite{berahas2024stochastic}.

After computing the normal component $v_k$, our algorithm computes the tangential component $u_k$ through an \mo{additional} constrained subproblem, \mo{involving the stochastic gradient estimate $g_k$}. We start by introducing the tangential subproblem with an assumption on $H_k$, which assumes \mo{that $H_k$ is} positive-definite in the null space of $J_k$ \mo{and is chosen independently of $g_k$}.
\begin{ass}\label{ass: assumption on Hk}
    For all $k\in \bN$, the matrix $H_k\in \bS^n$ is chosen independently from $g_k$, the sequence $\{H_k\}$ is bounded in norm by $\kappa_H$, and there exists $\zeta \in\bR_{>0}$, such that $u^TH_ku\ge \zeta \|u\|^2$ for all $u\in \text{Null}(J_k).$
\end{ass}

The subproblem to compute the tangential component $u_k$ is formulated as 
\begin{equation}\label{eq: subproblem u_k}
    \min_{u\in \bR^n} (g_k+H_kv_k)^Tu + \frac{1}{2}u^TH_ku \quad \text{ s.t.} \quad J_ku=0.
\end{equation}
Under Assumption \ref{ass: assumption on Hk}, $u_k$ is unique and can be obtained by solving the corresponding Newton system
\begin{equation}\label{eq: subproblem u_k's KKT system}
    \begin{bmatrix}
        H_k\quad J_k^T \\
        J_k \quad 0
    \end{bmatrix}  \begin{bmatrix}
        u_k \\ y_k
    \end{bmatrix}= -\begin{bmatrix}
        g_k +H_kv_k \\ 0
    \end{bmatrix},
\end{equation}
\mo{although the Lagrange multipliers $y_k$ may} have multiple solutions \mo{due to rank deficiency}. An exact solution to this subproblem can be obtained through \mo{factorization methods}, while inexact solutions can be obtained with iterative solvers. We will later discuss the \mo{possibility of inexactly computing $u_k$} in \Cref{sec: inexact algs}.

\mo{Throughout, we assume the following standard assumptions about the stochastic gradient estimate $g_k$}.
\begin{ass}\label{ass2}
    For all $k \in \bN,$ the stochastic gradient estimate $g_k\in \bR^n$ is an unbiased estimator of $\nabla f(x_k)$, \ie $\ \bE_k[g_k] = \nabla f(x_k)$, where $\bE_k[\cdot]$ denotes the conditional expectation up to iteration $x_k$. In addition, there exists $M\in \bR_{>0}$ such that $\bE_k[ \|g_k - \nabla f(x_k)\|^2] \le M$.
\end{ass}


After computing the two orthogonal components, a natural way to compute the step direction $d_k$ is to simply add them together, as in \cite{berahas2024stochastic}. However, our algorithm first rescales the tangential part $u_k$ \mo{by} some pre-defined parameter sequence $\{\beta_k\} > 0$ and then \mo{combines the components} to form $d_k$ instead. Formally, \mo{we set}
\begin{equation}\label{eq: compute dk=bkuk+vk}
    d_k = \beta_k u_k + v_k.
\end{equation}
The next iteration $x_{k+1}$ is then produced by $x_{k+1} = x_k +\alpha_k d_k$, where $\alpha_k$ is the second stepsize. The role of $\beta_k$ is crucial in our analysis. Simply put, $\{\beta_k\}$ controls the variance introduced by the stochastic gradient estimates $g_k$, \mo{but does not impact convergence in the constraints, which is driven by $v_k$}. Now we can state a general algorithm framework, where the subproblem \cref{eq: subproblem u_k} can be solved exactly or inexactly.
\begin{algorithm}
\caption{Two-Stepsize Stochastic SQP Algorithm Framework}
\label{algorithm1}
\begin{algorithmic}
\STATE Initialize $x_0 \in \bR^n$.
\FOR{$k=0, 1, 2, ..., K-1$}
    \STATE Compute the normal component $v_k$ with \cref{eq: subproblem v_k}.
    \STATE Compute a stochastic gradient estimate $g_k$.
    \STATE Compute the tangential component $u_k$ exactly by \cref{eq: subproblem u_k} or inexactly by \cref{eq: define rho_k and r_k}.
    \STATE Choose $\beta_k > 0$.
    \STATE Set $d_k \leftarrow \beta_k u_k + v_k$.
    \STATE Choose $\alpha_k > 0$.
    \STATE Set $x_{k+1} \leftarrow x_k + \alpha_k d_k$.
\ENDFOR
\end{algorithmic}
\end{algorithm}

There are various ways to choose the stepsizes $\{\gb_k\}$ and $\{\alpha_k\}$. Since \mo{the purpose of} $\{\gb_k\}$ \mo{is to control} the variance of stochastic gradient estimates, they are essentially equivalent to the stepsize \mo{employed} in stochastic gradient methods. \mo{Thus, in order} to obtain our \mo{desired complexity} result, we set $\gb_k = O(1/\sqrt{K})$ where $K$ is the total number of iterations that we plan to perform, which is a standard choice in the stochastic gradient literature \cite{ghadimi2013stochastic}. For the sake of brevity, we assign $\kappa_\beta\in\bR_{>0}$ as an upper bound such that $\gb_k \le \kappa_\beta$ for all $k\in\bN$.  

On the other hand, $\alpha_k$ can be chosen to be independent of $K$ because we do not require it to control the stochastic error \mo{and it can be viewed as an essentially deterministic step size, which should be inversely proportional to smoothness constants of the problem}. Specifically, we choose $\alpha_k$ from an interval \mo{involving} $\beta_k$ as follows:
\begin{equation}\label{alpha choice: interval}
    \alpha_k \in [\nu, \nu + \theta \beta_k],
\end{equation}
where $\nu \in\bR_{>0} $ \mo{satisfies properties with respect to problem-specific parameters} (see, \cref{eq: define nu}) and $\theta \in\bR_{>0}$ is a user-defined constant. This interval type of $\alpha_k$ is common in stochastic SQP methods, as in \cite{berahas2021sequential} and \cite{o2024two}.

\section{Convergence Analysis}\label{sec: convergence analysis}

To start our analysis, we first define some ``true" variables, which are quantities that would be computed if we used the full gradient $\nabla f(x_k)$ rather than stochastic gradient estimates $g_k$. For instance, let
\begin{equation}\label{eq: true subproblem u_k}
    \uktrue := 
\operatorname*{argmin}_{u\in \bR^n} \, (\nabla f(x_k)+H_kv_k)^Tu + \frac{1}{2}u^TH_ku \quad \text{s.t.} \quad J_ku=0.
\end{equation}
Equivalently, $\uktrue$ can be obtained by solving the ``true" Newton system
\begin{equation}\label{eq: true KKT system}
    \begin{bmatrix}
        H_k\quad J_k^T \\
        J_k \quad 0
    \end{bmatrix}  \begin{bmatrix}
        \uktrue \\ \yktrue
    \end{bmatrix}= -\begin{bmatrix}
        \nabla f(x_k)+H_kv_k \\ 0
    \end{bmatrix}.
\end{equation}
\mo{As mentioned previously, $\yktrue$ may not be uniquely defined by \eqref{eq: true KKT system}. Thus, throughout our analysis, we impose the choice of the least-squares solution, i.e.,
\begin{equation} \label{eq:yktrue}
    \yktrue := -(J_k^T)^+ (\nabla f(x_k) + H_k v_k + H_k \uktrue).
\end{equation}} Since $v_k$ is independent from the stochastic gradient estimate $g_k$, $v_k^{\text{true}}$ is identical to $v_k$. Therefore, $\dktrue = \beta_k\uktrue + v_k$, and \mo{we have the} useful \mo{property}: 
\begin{equation}\label{eq: Jkdktrue = Jkdk}
    J_k\dktrue = J_k(\beta_k \uktrue+v_k) = J_kv_k=J_kd_k.
\end{equation}

\subsection{Merit function, local model and preliminary results}

To study the convergence of SQP methods, a useful tool is the merit function. We employ the $l_2$ merit function $\phi:\bR^n\times \bR_{> 0} \to \bR$ defined by
\begin{equation}\label{eq: merit function}
    \phi(x, \tau) = \tau f(x) + \|c(x)\|
\end{equation}
in our analysis, where $\tau$ is the merit parameter.

We also introduce a local model of the merit function, $l: \bR^n\times\bR_{>0}\times \bR^n \times \bR^n \to \bR:$\begin{equation}\label{eq: local model l}
    l(x, \tau, g, d) = \tau(f(x) +g^Td) +\|c(x)+J(x)d\|.
\end{equation} 
The local model $l(x, \tau, g, d)$ can be seen as an approximation \mo{of $\phi$} near the current point $x$. \mo{As we will see, $l$ serves as a bridge} between the merit function $\phi$ and the infeasibility measure $\|J(x)^Tc(x)\|$. We also define the reduction in the local model for a direction $d$ and a gradient estimate $g$ as follows:
\begin{equation}\label{eq: model reduction Delta l}
\begin{aligned}
    \Delta l(x, \tau, g, d) &:= l(x, \tau, g, 0) - l(x, \tau, g, d) \\
    &= -\tau g^Td + \|c(x)\| - \|c(x) +J(x)d\|.
\end{aligned}
\end{equation}

Although \cref{algorithm1} does not involve any computation of the merit parameter $\tau$, it still plays an important role in complexity analysis. Similarly to \cite{berahas2024stochastic}, we assume that the sequence of merit parameters $\{\tau_k\}$ can be generated by the following scheme. The only difference is that we substitute values involving stochastic gradients with ``true" values and the rescaling parameter $\gb_k$.

For some fixed $\sigma \in (0, 1)$ and $\varepsilon_\tau \in [0, 1)$, let 
\small{\begin{equation}\label{eq: tau_k trial generating}
    \tau_k^{trial} \leftarrow \Bigg\{ 
    \begin{array}{ll}
        \infty & \text{if} \,\  \nabla f(x_k)^T \dktrue + \beta_k(\uktrue) ^T H_k\uktrue \le 0 \\
        \frac{(1-\sigma) (\|c_k\| - \|c_k+J_k\mo{v_k}\|)}{\nabla f(x_k)^T \dktrue + \beta_k(\uktrue)^TH_k \uktrue} & \text{otherwise},
    \end{array}
\end{equation}}
and
\begin{equation} \label{eq:taukupdate}
    \tau_k \leftarrow \Big\{ 
    \begin{array}{ll}
        \tau_{k-1} & \text{if} \,\  \tau_{k-1} \le \tau_k^{\text{trial}} \\
        \min \{(1-\eps_\tau) \tau_{k-1}, \tau_k^{\text{trial}} \} & \text{otherwise}.
    \end{array}
\end{equation}
Hence, the merit parameter sequence $\{\tau_k\}$ is monotonically non-increasing and satisfies 
\begin{equation}\label{eq:tau_k(true) leq reduction}
    \tau_k(\nabla f(x_k)^T \dktrue + \beta_k(\uktrue)^TH_k \uktrue) \le (1-\sigma)(\|c_k\| - \|c_k+J_k\mo{v_k}\|).
\end{equation}
From the fact that $\Delta l(x_k, \tau_k, \nabla f(x_k), \dktrue) = -\tau_k \nabla f_k^T\dktrue + (\|c_k\| - \|c_k +J_k\mo{v_k}\|)$ and $\sigma < 1$, we have the following model reduction inequality:
\begin{equation}\label{eq: local model reduction}
    \Delta l(x_k, \tau_k, \nabla f(x_k), \dktrue) \ge \tau_k \beta_k (\uktrue)^TH_k\uktrue + \sigma(\|c_k\| - \|c_k +J_k\mo{v_k}\|).
\end{equation}

Now we state some useful bounds derived from the step computation subproblems \cref{eq: subproblem v_k} and \cref{eq: subproblem u_k}. They link the local model reduction with our infeasibility measure. Notice that these lemmas \emph{do not} depend on the behavior of the merit parameter sequence and can be used for all cases we later discuss. The first lemma provides a tight bound on the infeasibility measure reduction. 
\begin{lemma}\label{lem: useful lemma from Byrd paper}
    For any one dimensional optimization problem
    \begin{equation*}
        \min_{z\in\bR^n} \Phi(z) = \frac{1}{2}a z^2 - bz\quad \text{s.t.} \quad z\le \omega,
    \end{equation*}
    where $b>0$ and $\omega>0$, its optimal value $\Phi^*$ satisfies
    $$
    \Phi^*\le -\frac{b}{2}\min\Big\{\frac{b}{|a|}, \omega\Big\}.
    $$
\end{lemma}
\begin{proof}
    See \cite[Lemma 2.1]{byrd2000trust}.
\end{proof}

We know that $\ga_k^C$ is the solution to the problem \mo{$\underset{\ga>0}{\min} \frac{1}{2}\|c_k+\ga J_kv_k^C\|^2 \ \text{s.t.} \ \alpha \le \omega$}. Using \cref{lem: useful lemma from Byrd paper}, setting $a=\|J_kJ_k^Tc_k\|^2$ and $ b=\|J_k^Tc_k\|^2$, we have \begin{equation}\label{eq: Phi star} \Phi^*(\ga) = \frac{1}{2}\|J_kJ_k^Tc_k\|^2 (\ga_k^C)^2 - \|J_k^Tc_k\|^2(\ga_k^C) \le -\frac{\|J_k^Tc_k\|^2} {2}\min\Big\{\frac{\|J_k^Tc_k\|^2}{\|J_kJ_k^Tc_k\|^2} , \omega \Big\}.
\end{equation}

\begin{lemma}\label{lemma: constraint model reduction}
    \mo{Let Assumption \ref{assumption1} hold.} Then, there exists $\kappa_v\in\bR_{>0}$ such that, for all $k\in \bN$ with $\|c_k\| > 0$,
    \begin{equation*}
        \|c_k\|(\|c_k\|-\|c_k+J_kv_k\|) \ge \kappa_v\|J_k^Tc_k\|^2.
    \end{equation*}
\end{lemma}
\begin{proof}
    We first consider the Cauchy step $v_k^C= -J_k^Tc_k$ of the subproblem \cref{eq: subproblem v_k}. Using the previous conclusion \cref{eq: Phi star}, we have
    \begin{equation}
        \begin{aligned}
            \frac{1}{2}(\|c_k + \alpha_k^CJ_kv_k^C\|^2 - \|c_k\|^2) &= \frac{1}{2}(2\ga_k^Cc_k^TJ_kv_k^C + (\alpha_k^C)^2\|J_kv_k^C\|^2)\\ 
            & = -\ga_k^C\|J_k^Tc_k\|^2 + \frac{1}{2}(\ga_k^C)^2\|J_kJ_k^Tc_k\|^2 \\
            & \le -\frac{1}{2}\|J_k^Tc_k\|^2\min\Big\{\frac{\|J_k^Tc_k\|^2}{\|J_kJ_k^Tc_k\|^2} , \omega \Big\} \\
            & \le -\frac{1}{2}\|J_k^Tc_k\|^2\min\Big\{\frac{1}{\|J_k^TJ_k\|} , \omega \Big\}.
        \end{aligned}
    \end{equation}
    Now, since $v_k$ satisfies the Cauchy decreasing condition \cref{eq: cauchy decreasee condition}, we have the following inequality with the fact $x(x-y) \ge \frac{1}{2}(x^2-y^2)$ for any $x, y \in \bR$ that 
    \begin{equation}
        \begin{aligned}
            \|c_k\|(\|c_k\| - \|c_k+J_kv_k\|) &\ge \epsilon_v \|c_k\|(\|c_k\| - \|c_k+\alpha_k^C J_kv_k^C\|) \\
            & \ge \frac{1}{2}\epsilon_v (\|c_k\|^2 - \|c_k + \alpha_k^CJ_kv_k^C\|^2) \\
            &\ge \frac{\epsilon_v}{2}\|J_k^Tc_k\|^2 \min\{\go, \frac{\go^2}{\|J_k^TJ_k\|}\}.
        \end{aligned}
    \end{equation}
    Since $\|J_k^TJ_k\|$ is bounded from above by Assumption \ref{assumption1}, there exists a constant $\kappa_v \in (0, \min\{\go, \frac{\go^2}{\|J_k^TJ_k\|}\})$ such that 
    \begin{equation*}
         \|c_k\|(\|c_k\|-\|c_k+J_kv_k\|) \ge \kappa_v\|J_k^Tc_k\|^2,
    \end{equation*}
    which proves the result.
\end{proof}

\begin{lemma}\label{lemma: norm vk}
    There exists $\underline{\go} \in\bR_{>0}$ such that, for all $k\in\bN$ with $\|c_k\|>0$,
    $$
    \underline{\go}\|J_k^Tc_k\|^2 \le \|v_k\|\le \go \|J_k^Tc_k\|.
    $$
\end{lemma}
\begin{proof}
    The right inequality comes from the trust-region constraint of \cref{eq: subproblem v_k}. From the triangle inequality and the Cauchy-Schwarz inequality, we have 
    $$
    \|c_k\| - \|c_k+J_kv_k\| \le \|J_kv_k\| \le \|J_k\|\|v_k\|.
    $$
    If $\|J_k\| = 0$ \mo{or $\|c_k\| = 0$}, the desired inequality is trivial. Otherwise, by \cref{lemma: constraint model reduction},
    \begin{equation*}
        \|v_k\|\ge \frac{ \|c_k\| - \|c_k+J_kv_k\|}{\|J_k\|}\ge \frac{\kappa_v\|J_k^Tc_k\|^2}{\|J_k\|\|c_k\|}.
    \end{equation*}
    Since $J_k$ and $c_k$ are assumed to be bounded in norm, there exists some $\underline{\go} > 0$ that satisfy $\|v_k\|\ge \underline{\go}\|J_k^Tc_k\|^2$ for all $k$.
\end{proof}

\cref{lemma: constraint model reduction} and \cref{lemma: norm vk} are \mo{are very useful in the Byrd-Omojokun type of methods \cite{omojokun1989trust} when constraint qualifications do not hold}. \mo{We additionally prove stronger versions of them when LICQ holds in \Cref{sec: when LICQ holds}}. \mo{Next, we consider a bound on the ``true" tangential component, $\uktrue$}.

\begin{lemma}\label{lemma2: u_k norm}
    Let Assumptions \ref{assumption1} and \ref{ass: assumption on Hk} hold. Let $\yktrue$ be defined by \eqref{eq:yktrue} for all $k\in\bN$. Then, we have
    \begin{equation*}
        \|\uktrue\| \le \zeta^{-1} (\|\nabla f(x_k) + J_k^T \yktrue\| + \kappa_H \omega \|J_k^T c_k\|),
    \end{equation*}
    and
    $$\|\uktrue\|\le \zeta^{-1}(\kappa_g + \kappa_H\omega \kappa_J \kappa_c) = : \kappa_u.$$
\end{lemma}
\begin{proof}
    From \cref{eq: true KKT system}, we know $H_k\uktrue + J_k^T\yktrue=-\nabla f(x_k) + H_k v_k$. By Assumption \ref{ass: assumption on Hk}, since $\uktrue$ lies strictly in the null space of $J_k$, we have
    \begin{align*}
        \zeta\|\uktrue\|^2 & \le (\uktrue)^T H_k \uktrue \\
        & = -\nabla f(x_k)^T\uktrue - v_k^T H_k \uktrue \\
        & = -(\nabla f(x_k) + J_k^T \yktrue)^T\uktrue - v_k^T H_k \uktrue \\
        &\le \|\nabla f(x_k) + J_k^T \yktrue\| \|\uktrue\| + \kappa_H \|v_k\| \|\uktrue\| \\
        &\le \|\uktrue\| (\|\nabla f(x_k) + J_k^T \yktrue\| + \kappa_H \omega \|J_k^T c_k\|).
    \end{align*} 
    Dividing both sides by $\zeta \|\uktrue\|$ yields the first result. Similarly,
    \begin{align*}
        \zeta\|\uktrue\|^2 &\le -\nabla f(x_k)^T\uktrue - v_k^T H_k \uktrue \\
        & \le \|\nabla f(x_k)\|\|\uktrue\| 
        + \kappa_H \|v_k\|\|\uktrue\| \\
        & \le \|\uktrue\|(\kappa_g + \kappa_H \omega \|J_k^Tc_k\|) \\
        & \le \|\uktrue\|(\kappa_g + \kappa_H\omega \kappa_J\kappa_c).
    \end{align*}
    As before, dividing through yields the second result.
\end{proof}
The following lemma bounds the distance in expectation between the computed direction $d_k$ and the true direction $\dktrue$, which reflects the variance of the stochastic estimates.

\begin{lemma}\label{lemma: reduced KKT}
    Let Assumptions \ref{ass: assumption on Hk} and \ref{ass2} hold. Then, for all $k\in \bN$, $\bE_k [u_k] = \uktrue$, $\bE_k[d_k] = \dktrue$, 
    and $\bE_k[\|d_k-\dktrue\|]\le \beta_k \zeta^{-1}\sqrt{M}$.
\end{lemma}
\begin{proof}
    Under Assumption \ref{ass: assumption on Hk}, there exists a matrix $Z_k$ whose columns form an orthogonal basis for $\text{Null}(J_k)$, and vectors $w_k, \wktrue$ such that $u_k = Z_kw_k$ and $\uktrue = Z_k\wktrue$. From subproblem \cref{eq: subproblem u_k}, \mo{explicit forms of $w_k$ and $\wktrue$ are given as}:
    $$
    w_k = -(Z_k^TH_kZ_k)^{-1}Z_k^T (g_k +H_kv_k); \quad \wktrue = -(Z_k^TH_kZ_k)^{-1}Z_k^T (\nabla f(x_k) +H_kv_k).
    $$
    Since $(Z_k^TH_kZ_k)^{-1}Z_k^T$ and $Z_k$ are both linear operators, from Assumption \ref{ass2} and the property of expectation, we have $\bE_k [ u_k] = \uktrue$. \mo{In addition, since $\beta_k$ is independent of $g_k$}, $\bE_k[d_k] = \bE_k[\beta_ku_k +v_k] = \beta_k\bE_k[u_k] +v_k = \beta_k \uktrue +v_k = \dktrue$.
    
     The rest of the proof follows from Jensen's inequality and Assumptions \ref{ass: assumption on Hk} and \ref{ass2}.
        \begin{align*}
            \bE_k[\|d_k - \dktrue\|] & = \bE_k [\|\beta_k (u_k - \uktrue)\|]  \\
            &= \gb_k\bE_k[\|Z_k(w_k - \wktrue)\|] \\
            & = \gb_k \bE_k [\|Z_k(Z_k^TH_kZ_k)^{-1}Z_k^T (g_k -\nabla f(x_k)\|] \\
            & \le \beta_k \bE_k [\|Z_k(Z_k^TH_kZ_k)^{-1}Z_k^T\| \|g_k - \nabla f(x_k)\|] \\
            & = \beta_k \|(Z_k^TH_kZ_k)^{-1}\| \bE_k [\|g_k - \nabla f(x_k)\| ] \\
            & \le \gb_k\zeta^{-1} \sqrt{\bE_k [\|g_k - \nabla f(x_k)\|^2]} \\
            &\le \gb_k\zeta^{-1}\sqrt{M}.
        \end{align*}

\end{proof}

Next, we provide an important upper bound on the expectation of $\|u_k\|^2$.

\begin{lemma}\label{lem: upper bound of norm u_k}
Let Assumptions \ref{ass: assumption on Hk} and \ref{ass2} hold. Then, for all $k \in \bN,$ 
$$\bE_k [\|u_k\|^2] \le \zeta^{-2}M + \zeta^{-1} (\uktrue)^T H_k \uktrue.$$
\end{lemma}
\begin{proof} Since $u_k, \uktrue\in \text{Null}(J_k)$, by Assumption \ref{ass: assumption on Hk} and \cref{eq: subproblem u_k's KKT system}, we have
    \begin{align*}
        & \quad \ \bE_k [\|u_k\|^2] - \zeta^{-1}(\uktrue)^T H_k\uktrue \\
        &\le \zeta^{-1}( \bE_k [u_k^TH_ku_k] - (\uktrue)^T H_k\uktrue) \\
        & = \zeta^{-1} (\bE_k[u_k^T(-u_k^TJ_k^Ty_k-g_k-H_kv_k)] -(\uktrue)^T(-(\uktrue)^TJ_k^T\yktrue - \nabla f(x_k)-H_kv_k)\\
        & = \zeta^{-1} (\nabla f(x_k)^Tu_k -\bE_k [g_k^Tu_k]).
    \end{align*}
    Since $u_k = Z_kw_k = Z_k(Z_k^TH_kZ_k)^{-1}Z_k^T (g_k + H_kv_k)$, we know 
    \begin{align*}
        \bE_k[g_k^Tu_k] &= \bE_k[-g_k^TZ_k(Z_k^TH_kZ_k)^{-1}Z_k^T (g_k + H_kv_k)] \\ 
        &\le -\zeta^{-1}\bE_k[\|g_k^TZ_k\|^2] - \zeta^{-1}\bE_k[g_k^TH_kv_k] \\
        &= -\zeta^{-1}(\bE_k[\|g_k^TZ_k\|^2] + \nabla f(x_k)^TH_kv_k).
    \end{align*}
    Similarly, $\nabla f(x_k)^Tu_k \le -\zeta^{-1} (\|g_k^TZ_k\|^2 + \nabla f(x_k)^TH_kv_k).$ Combining the inequalities and Assumption \ref{ass2}, we have 
    \begin{equation*}
        \bE_k[\|u_k\|^2]-\zeta^{-1}(\uktrue)^TH_k\uktrue) \le \zeta^{-2} (-\|Z_k^Tg_k\|^2 + \bE_k [\|Z_k^Tg_k\|^2]) \le \zeta^{-2}M,
    \end{equation*}
    which proves the result.
\end{proof}

\mo{The following lemma allows us to connect \eqref{eq: local model reduction} with the gradient of the Lagrangian.}

\begin{lemma}\label{lemma: uHutrue with KKT measure no LICQ}
    Let Assumptions \ref{assumption1} and \ref{ass: assumption on Hk} hold and let $\yktrue$ be defined by \eqref{eq:yktrue}. Then,
    $$
    (\uktrue)^TH_k(\uktrue) \ge \frac{\zeta }{2\kappa_H^{2}} \|\nabla f(x_k) + J_k^Ty_k^\text{true}\|^2 - \kappa_0 \|J_k^T c_k\|^2,
    $$
    where $\kappa_0:=\omega^2(2 \kappa_H +2 \zeta^{-1}\kappa_H^2 + \zeta)$.
\end{lemma}
\begin{proof}
    From Assumption \ref{ass: assumption on Hk}, we know
    \begin{equation*}
        (\uktrue)^TH_k(\uktrue) \ge \zeta \|\uktrue\|^2 \ge \zeta \kappa_H^{-2} \|H_k\uktrue\|^2.
    \end{equation*}
    Then, from \cref{eq: true KKT system},  \cref{lemma: norm vk}, and \cref{lemma2: u_k norm}, we can conclude that
    \begin{align*}
        \|H_k\uktrue\|^2 &=\|H_k\uktrue + H_kv_k - H_kv_k\|^2 \\
        &= \|H_k\uktrue + H_kv_k\|^2 - 2v_k^TH_kH_k(\uktrue+v_k) + \|H_kv_k\|^2 \\
        & \ge \| \nabla f(x_k) + J_k^T\yktrue\|^2 - 2v_k^TH_kH_k\uktrue - \|H_kv_k\|^2 \\
        &\ge \| \nabla f(x_k) + J_k^T\yktrue\|^2 - \kappa_H^2 (2\|\uktrue\|+\|v_k\|)\|v_k\|\\
        &\ge \| \nabla f(x_k) + J_k^T\yktrue\|^2 \\
        &\qquad- \kappa_H^2 (2\zeta^{-1}\|\nabla f(x_k) + J_k^T\yktrue\| + 2\zeta^{-1} \kappa_H \omega \|J_k^T c_k\| + \|v_k\|)\|v_k\| \\
        &\ge \| \nabla f(x_k) + J_k^T\yktrue\|^2 - 2\omega \zeta^{-1} \kappa_H^2\|\nabla f(x_k) + J_k^T\yktrue\|\|J_k^T c_k\| \\
        &\qquad- \kappa_H^2\omega^2 (2\zeta^{-1} \kappa_H +1) \|J_k^T c_k\|^2.
    \end{align*}
    Next, applying Young's inequality to the 2nd term in the above inequality,
    \begin{equation*}
        2\omega \zeta^{-1} \kappa_H^2\|\nabla f(x_k) + J_k^T\yktrue\|\|J_k^T c_k\| \le \frac12 \|\nabla f(x_k) + J_k^T\yktrue\|^2 + 2 \omega^2 \zeta^{-2} \kappa_H^4\|J_k^T c_k\|^2,
    \end{equation*}
    and thus
    \begin{equation*}
        \|H_k\uktrue\|^2 \ge \frac12\| \nabla f(x_k) + J_k^T\yktrue\|^2 - \kappa_H^2\omega^2(2\zeta^{-1} \kappa_H +2 \zeta^{-2}\kappa_H^2 + 1)\|J_k^T c_k\|^2,
    \end{equation*}
    which proves the result.
\end{proof}

    

\mo{Recall that} the merit parameter sequence $\{\tau_k\}$ is monotonically non-increasing and non-negative. Therefore, by the monotone convergence theorem, there are two possible tail behaviors: 
\begin{enumerate}[label=(\roman*)]
    \item There exists some $\taum \in \bR_{>0}$ such that $\lim_{k\to \infty} \tau_k = \tau_{\min}$,
    \item $\{\tau_k\}$ vanishes, \ie \ $\lim_{k\to \infty} \tau_k = 0$. 
\end{enumerate}
\mo{We formalize event (i) in the following assumption.
\begin{ass} \label{ass:taumin}
    There exists $\taum > 0$ such that
    \begin{equation} \label{eq:taumin}
            \taum(\nabla f(x_k)^T \dktrue + \beta_k(\uktrue)^TH_k \uktrue) \le (1-\sigma)(\|c_k\| - \|c_k+J_kv_k\|),
    \end{equation}
    holds for all $k \in \mathbb{N}$.
\end{ass}}

\mo{As an immediate consequence of this assumption}, by \cref{eq: local model reduction} we know that for all $k\in\bN$,
\begin{equation}\label{eq: model reduction and ukHkuk}
     \Delta l(x_k, \taum, \nabla f(x_k), \dktrue) \ge \taum \beta_k (\uktrue)^TH_k\uktrue + \sigma(\|c_k\| - \|c_k +J_k\mo{v_k}\|).
\end{equation}
This generic bound on the model reduction frees us from computing $\tau_k$ values, thus avoiding the over-complication of the analysis and extra computation in practice, unlike in \cite{berahas2024stochastic} \mo{which relies on stochastic estimates of $\tau_k$ at each iteration}.

Event (ii) leads the merit function $\phi$ to ignore the objective $f$ in \mo{the limit} and to minimize the constraint violation $\|c_k\|$ only. \mo{Such behavior can occur even in \textit{deterministic} SQP methods when $\{J_k\}$ does not have full rank, in which case the best one can hope for is convergence to an infeasible stationary point, which we prove below}.

We \mo{consider the convergence of \Cref{algorithm1} and associated complexity results under these different events}. In fact, \mo{Assumption \ref{ass:taumin} could hold regardless of whether LICQ holds}, which \mo{leads} to different complexity results \mo{in terms of} the infeasibility measure. Therefore, we will mainly discuss \emph{three} cases in the following subsections.

\subsection{Case I: \mo{Assumption \ref{ass:taumin} holds without constraint qualifications}} \label{sec: when taum and no licq}

In this subsection, we assume that $\tau_k \searrow \taum >0$ without any constraint qualifications. We first state a generic merit function descent lemma that relates the improvement of the merit function to the local model.

\begin{lemma}\label{lemma: generic descent}
    Let Assumptions \ref{assumption1}, \ref{ass: assumption on Hk}, \ref{ass2}, \mo{and \ref{ass:taumin}} hold, and $\alpha_k$ is chosen as in \cref{alpha choice: interval}. Then if $\alpha_k \le 1$, we have 
    \begin{equation}\label{eq: phi reduction}
        \begin{aligned}
        & \quad \ \phi(x_k+ \alpha_k d_k, \tau_{\min}) - \phi(x_k, \taum) \\ & \le  -\alpha_k \Delta l(x_k, \taum, \nabla f(x_k), \dktrue) + \frac{\alpha_k^2\beta_k^2}{2}(\taum L + \Gamma) \|u_k\|^2 \\ 
        & \qquad + \frac{\alpha_k^2}{2}(\taum L + \Gamma)\|v_k\|^2  +\alpha_k \taum \nabla f(x_k)^T(d_k - \dktrue).
    \end{aligned}
    \end{equation}
\end{lemma}
\begin{proof}
    By the $L$-Lipschitz continuity of $\nabla f(x)$ and $\Gamma$-Lipschitz continuity of $J_k$, we have 
    \begin{align} \label{generic descent of phi}
    \notag &\quad \ \phi(x_k+\alpha_k d_k, \taum) - \phi (x_k, \taum) \\ \notag
        & = \taum [f(x+\alpha_k d_k) - f(x_k)] + \|c(x_k+\alpha_kd_k)\| - \|c(x_k)\|  \\ \notag
        & \le \alpha_k\taum \nabla f(x_k)^Td_k + \alpha_k\|c_k+ J_kd_k\| + |1-\alpha_k|\|c_k\| \\ \notag
        & \qquad - \|c_k\| + \frac{1}{2}(\taum L+\Gamma) \alpha_k^2\|d_k\|^2 \\
        & = \alpha_k\taum \nabla f(x_k)^Td_k -\alpha_k (\|c_k\|-\|c_k+ J_kd_k\|) + \frac{1}{2}(\taum L+\Gamma) \alpha_k^2\|d_k\|^2.
    \end{align}
Since $d_k = \beta_k u_k + v_k$ and $u_k, v_k$ are orthogonal, we have $\|d_k\|^2 = \beta_k^2\|u_k\|^2 + \|v_k\|^2$. Combining with the fact $J_kd_k = J_k\dktrue$, we have 
\begin{align*}
    & \quad \ \phi (x_k +\alpha_k d_k, \taum) - \phi(x_k, \taum) \\
    & \le \alpha_k\taum \nabla f(x_k)^T \dktrue - \alpha_k ( \|c_k\| - \|c_k+J_k\dktrue\|) \\
    & \qquad + \frac{{\alpha_k^2}}{2}(\taum L + \Gamma)(\beta_k^2\|u_k\|^2 + \|v_k\|^2) + \alpha_k\taum \nabla f(x_k)^T (d_k - \dktrue) \\
    & \le -\alpha_k \Delta l(x_k, \taum, \nabla f(x_k), \dktrue) + \frac{\alpha_k^2\beta_k^2}{2}(\taum L + \Gamma) \|u_k\|^2  \\
    & \qquad + \alpha_k \taum \nabla f(x_k)^T(d_k - \dktrue) + \frac{\alpha_k^2}{2}(\taum L + \Gamma)\|v_k\|^2.
\end{align*}
\end{proof}

\cref{lemma: generic descent} relates the reduction in the merit function $\phi$ to the local model $l$, the norms of $v_k$ and $u_k$, and the stochastic error in $d_k$. It also reveals how $\beta_k^2$ acts on $\|u_k\|^2$ to \mo{control} the stochastic variance.

Recall that $\alpha_k$ is chosen from the interval $[\nu, \nu + \theta \beta_k]$. Now we choose \begin{equation}\label{eq: define nu}
\nu \in \left(0,  \min\left\{ \frac{\sigma\kappa_v\kappa_c^{-1}}{2(\taum L+\Gamma)\omega^2}, 1-\theta \kappa_\beta\right\}\right],
\end{equation} which ensures $\ga_k \le 1$. A natural drawback of this step size choosing scheme is that we must have knowledge of the Lipschitz constants $\Gamma, L$, and the bound of the merit parameter $\taum$, which \mo{may be} unreasonable in practice. Practically, one could consider adaptive parameter choosing schemes such as AdaGrad or \at{Adam} style stepsizes, which will yield similar complexity results with additional logarithmic terms \mo{\cite{o2024two, wang2026projected}}.

\begin{lemma}\label{lem: 3.9}
Let Assumptions \ref{assumption1}, \ref{ass: assumption on Hk}, \ref{ass2}, \mo{and \ref{ass:taumin}} hold, then 
\begin{equation*}
    \bE_k[\alpha_k \taum \nabla f(x_k)^T(d_k - \dktrue)] \le \beta_k^2\theta \taum \kappa_g \zeta^{-1}\sqrt{M}.
\end{equation*}
\end{lemma}
\begin{proof}
    Let $\xi_k \in[0, 1]$ be a random variable such that $\ga_k = \nu +\theta\xi_k\beta_k$. Then from \cref{lemma: reduced KKT},
    \begin{align*}
        & \quad \ \bE_k[\alpha_k \taum \nabla f(x_k)^T(d_k - \dktrue)] \\
        & = \bE_k [(\nu +\theta\xi_k\beta_k)\taum \nabla f(x_k)^T(d_k - \dktrue)] \\
        & = \nu \taum\nabla f(x_k)^T\bE[d_k - \dktrue] + \bE_k[\theta \xi_k \gb_k\taum \nabla f(x_k)^T(d_k - \dktrue)] \\ 
        & \le \theta\taum \bE_k[\|\nabla f(x_k)\|\|d_k-\dktrue\|] \\
        & \le \beta_k^2\theta\taum \kappa_g \zeta^{-1} \sqrt{M} .
    \end{align*}
\end{proof}

With all \mo{of these results} in hand, we can prove our first main theorem, which is stated generically for any choice of $\beta_k$.

\begin{theorem} \label{thm: taum exists convergence}
    Let Assumptions \ref{assumption1}, \ref{ass: assumption on Hk}, \ref{ass2}, \mo{and \ref{ass:taumin}} hold. Choose $\nu$ as in \cref{eq: define nu}, $\ga_k$ as in \cref{alpha choice: interval} and let 
\begin{align*}
\kappa_1 := \frac{1}{2}(\taum L + \Gamma)(\zeta^{-1} \kappa_H\kappa_u^2 + \zeta^{-2} M) + \theta \kappa_g \taum \zeta^{-1} \sqrt{M} + \theta^2 (\taum L +\Gamma) \omega \kappa_J\kappa_c.
\end{align*}
Then we have, for all $K\in \bN$,
\begin{align*}
    &\quad \sum_{k=0}^{K-1} \bE [\alpha_k\beta_k\taum (\uktrue)^TH_k\uktrue + \frac{1}{2} {\alpha_k\sigma\kappa_v\kappa_c^{-1}} \|J_k^Tc_k\|^2] \\ 
    & \le \taum(f_0 -  f_{\inf})+ \taum \|c_0\|+ \kappa_1 \sum_{k=0}^{K-1}\beta_k^2.
\end{align*}
\end{theorem}

\begin{proof}
    Taking the conditional expectation \mo{of the $k$-th iterate} on both sides of \eqref{eq: phi reduction} \mo{and applying the results of Lemmas \ref{lem: upper bound of norm u_k} and \ref{lem: 3.9}}, we have 
    \begin{align*}
        &\quad \ \bE_k[\phi (x_k +\alpha_k d_k, \taum)] - \phi(x_k, \taum) \\
        & \le \bE_k[-\alpha_k \Delta l(x_k, \taum, \nabla f(x_k), \dktrue)] +\frac{\alpha_k^2\beta_k^2}{2}(\taum L + \Gamma) \bE_k[\|u_k\|^2] \\
        & \qquad+\frac{\alpha_k^2}{2}(\taum L + \Gamma)\|v_k\|^2 + \bE_k[\alpha_k \taum \nabla f(x_k)^T(d_k - \dktrue)] \\
        & \le  -\alpha_k \Delta l(x_k, \taum, \nabla f(x_k), \dktrue)  + (\theta \kappa_g \taum \zeta^{-1} \sqrt{M}) \beta_k^2 + \frac{\alpha_k^2}{2}(\taum L + \Gamma) \|v_k\|^2 \\
        & \qquad+ \frac{\alpha_k^2\beta_k^2}{2}(\taum L + \Gamma)(\zeta^{-1} (\uktrue)^T H_k \uktrue + \zeta^{-2} M) \\
        & \le -\alpha_k \Delta l(x_k, \taum, \nabla f(x_k), \dktrue)  + (\nu^2 + \theta^2 \beta_k^2) (\taum L + \Gamma) \|v_k\|^2 \\ 
        & \qquad+ \beta_k^2 [\frac{\alpha_k^2}{2}(\taum L + \Gamma)(\zeta^{-1} (\uktrue)^T H_k \uktrue + \zeta^{-2} M) + \theta \kappa_g \taum \zeta^{-1} \sqrt{M}]\\
        & \le  -\alpha_k \Delta l(x_k, \taum, \nabla f(x_k), \dktrue) + \beta_k^2 \kappa_1 + \nu^2 (\taum L + \Gamma) \|v_k\|^2.
    \end{align*}
From \cref{eq: model reduction and ukHkuk} and Lemma \ref{lemma: constraint model reduction}, we have
    \begin{align}
        & \quad \ \bE_k[\phi (x_k +\alpha_k d_k, \taum)] - \phi(x_k, \taum) \notag \\
        & \le -\alpha_k (\tau_{\min}\beta_k (\uktrue)^TH_k\uktrue + \sigma (\|c_k\| - \|c_k+J_k\mo{v_k}\|)\mo{)} \nonumber \\
        & \qquad+ \beta_k^2 \kappa_1  + \nu^2 (\taum L + \Gamma) \|v_k\|^2  \notag \\
        & \le -\alpha_k\tau_{\min}\beta_k (\uktrue)^TH_k\uktrue + \beta_k^2\kappa_1  \nonumber \\ 
        & \qquad + (-\alpha_k\sigma + \nu^2(\taum L+\Gamma)\frac{\omega^2\kappa_c}{\kappa_v})(\|c_k\| - \|c_k+J_k\mo{v_k}\|)  \notag \\
        & \quad \Big( \text{Since} \,  \nu\le \min{\Big\{ \frac{\sigma\kappa_v}{2(\taum L+\Gamma)\omega^2\kappa_c}, 1-\theta\kappa_\beta
        \Big\}}\Big) \notag \\
        & \le -\alpha_k\tau_{\min}\beta_k (\uktrue)^TH_k\uktrue + \beta_k^2\kappa_1-\frac{1}{2}{\alpha_k\sigma}(\|c_k\| - \|c_k+J_k\mo{v_k}\|) \label{eq: used in LICQ equation} \\
        & \le -\alpha_k\tau_{\min}\beta_k (\uktrue)^TH_k\uktrue + \beta_k^2\kappa_1-\frac{1}{2}{\alpha_k\sigma\kappa_v}\kappa_c^{-1}\|J_k^Tc_k\|^2. \notag
    \end{align}
    Taking the total expectation, rearranging and summing the inequalities from $k=0$ to $K-1$,
    \begin{align*}
        & \quad \sum_{k=0}^{K-1} \bE [\alpha_k\beta_k\taum (\uktrue)^TH_k\uktrue + \frac{1}{2}\alpha_k\sigma\kappa_v\kappa_c^{-1}\|J_k^Tc_k\|^2] \\
        & \le \phi(x_0, \taum) - \bE[\phi(x_K, \taum)]+ \kappa_1 \sum_{k=0}^{K-1} \beta_k^2.
    \end{align*}
    Since $-\bE[\phi(x_K, \taum)]=-\bE [\taum f(x_K)+\|c_K\|] \le -\taum f_{\inf}$, we can conclude that 
    \begin{align*}
        &\quad \sum_{k=0}^{K-1} \bE [\alpha_k\beta_k\taum (\uktrue)^TH_k\uktrue + \frac{1}{2}\alpha_k\sigma\kappa_v\kappa_c^{-1}\|J_k^Tc_k\|^2] \\ 
        & \le \phi(x_0, \taum) - \taum f_{\inf}+ \kappa_1 \sum_{k=0}^{K-1}\beta_k^2.
    \end{align*}
\end{proof}

\Cref{thm: taum exists convergence} shows that \mo{under Assumption \ref{ass:taumin}} and $\sum\beta_k^2 < \infty$, the infeasibility measure $\|J_k^Tc_k\|$ converges in expectation. The convergence rate is determined by problem-specific constants, such as $L$ and $\Gamma$, and the starting point $x_0$. In fact, if the exact order of $\{\gb_k\}$ is known, we can obtain a worst-case complexity for $\bE\|J_k^Tc_k\|$ \mo{as well as $\bE\|\nabla f(x_k) + J_k^T \yktrue\|$}.

\begin{corollary}\label{cor: Jkck2 complexity}
    \mo{Let the assumptions of \Cref{thm: taum exists convergence} hold.} For any $K\in \bN$, let $\beta_k := \beta = \eta / \sqrt K $ for all $k\in[0, K-1]$, where $\eta >0$. Let $\kappa_2 := \taum (f_0 - f_{\inf} ) +\|c_0\| + \eta^2 \kappa_1$, then 
    $$
    \frac{1}{K} \sum_{k=0}^{K-1} \bE [\|J_k^T c_k\|^2] \le \frac{2\kappa_c\kappa_2}{\nu \sigma \kappa_v K}.
    $$
\end{corollary}

\begin{proof}
    From \cref{thm: taum exists convergence}, the definition of $\beta_k$, \mo{Assumption \ref{ass: assumption on Hk}}, and the fact that $\nu \le \alpha_k$ for all $k\in\bN$, we have 
    \begin{align*}
        \frac{1}{K} \sum_{k=0}^{K-1} \bE [\|J_k^T c_k\|^2] \le \frac{2\kappa_c(\taum (f_0 - f_{\inf}) + \|c_0\| + \kappa_1 \sum \beta_k^2)}{\nu \sigma \kappa_v K} \le \frac{2\kappa_c\kappa_2}{\nu\sigma\kappa_v K}.
    \end{align*}
\end{proof}

\begin{corollary}\label{cor: reduced gradient complexity}
    \mo{Let the assumptions of \Cref{thm: taum exists convergence} hold.} For any $K\in\bN$, let $\beta_k := \beta = \eta / \sqrt K $ for all $k\in[0, K-1]$, where $\eta >0$. Let $\yktrue$ be defined in \eqref{eq:yktrue}. Then,
    \begin{equation} \label{eq:complgradl}
        \frac{1}{K}\sum_{k=0}^{K-1} \bE[\|\nabla f(x_k) + J_k^Ty_k^\text{true}\|^2] \le \frac{2\kappa_H^2\kappa_2}{\nu\eta\taum\zeta\sqrt{K}} + \frac{4\kappa_H^2 \kappa_c\kappa_0\kappa_2}{\nu \sigma\kappa_v \zeta K}.
    \end{equation}
\end{corollary}
\begin{proof}
 Similarly to \cref{cor: Jkck2 complexity}, we know 
    \begin{equation*}
        \sum_{k=0}^{K-1} \bE [\alpha_k\beta_k\taum (\uktrue)^TH_k\uktrue] \le \phi(x_0, \taum) - \taum f_{\inf}+ \kappa_1 \sum_{k=0}^{K-1}\beta_k^2.
    \end{equation*}
    From \cref{lemma: uHutrue with KKT measure no LICQ}, it follows that
    \begin{align*}
        & \quad \ \sum_{k=0}^{K-1} \bE [\frac{\alpha_k\beta_k\taum \zeta }{2\kappa_H^{2}} \|\nabla f(x_k) + J_k^Ty_k^\text{true}\|^2] \\
        & \le \phi(x_0, \taum) - \taum f_{\inf}+ \kappa_1 \sum_{k=0}^{K-1}\beta_k^2 + \kappa_0 \sum_{k=0}^{K-1} \bE[\alpha_k\beta_k\taum \|J_k^T c_k\|^2].
    \end{align*}
    Substituting $\beta_k $ with $\eta/\sqrt{K}$, rearranging the inequality, dividing both sides by $K$, and applying \cref{cor: Jkck2 complexity}, we have 
    \begin{equation*}
        \frac{1}{K}\sum_{k=0}^{K-1} \bE[\|\nabla f(x_k) + J_k^Ty_k^\text{true}\|^2] \le \frac{2\kappa_H^2\kappa_2}{\nu\eta\taum\zeta\sqrt{K}} + \frac{4\kappa_H^2 \kappa_c\kappa_0\kappa_2}{\nu \sigma\kappa_v \zeta K},
    \end{equation*}
    \mo{which proves the result.}
\end{proof}

\Cref{cor: Jkck2 complexity} and \cref{cor: reduced gradient complexity} \mo{prove our first set of} complexity results. It is obvious that for any small tolerance $\epsilon_c, \epsilon_L>0$, to achieve $\bE\|J_k^Tc_k\| \le \epsilon_c$, we need at most $\mathcal{O}(\epsilon_c^{-2})$ iterations. To achieve \mo{$\bE\|\nabla f(x_k) + J_k^T \yktrue\|\le \epsilon_L$}, we need at most $\mathcal{O}(\epsilon_L^{-4})$ iterations. To our knowledge, \mo{the complexity with respect to feasibility represents the best known result in the absence of a constraint qualification}. The latter result, \mo{in terms of $\epsilon_L$}, is known to be \emph{optimal} \mo{\cite{arjevani2023lower}} for stochastic gradient methods to find first-order stationary points under our assumptions.

\subsection{Case II: LICQ holds}\label{sec: when LICQ holds}
In this subsection, we assume that the Linear Independence Constraint Qualification is satisfied for every iteration. The analysis is parallel to Section \ref{sec: when taum and no licq}, however, under this assumption, we are able to prove tighter bounds with respect to convergence in the infeasibility measure and \at{the gradient of the Lagrangian}. It is worth mentioning that most of the existing literature on \mo{complexity results in} constrained optimization assumes LICQ or other types of constraint qualification in their analysis, such as \mo{variants of MFCQ} \cite{boob2025level, shi2025momentum}. We formally state LICQ as an assumption here.

\begin{ass}\label{ass: LICQ}(LICQ)
    For all $k\in\bN$, the constraint Jacobian matrix $J(x_k)$ has full row rank. Equivalently, there exists a positive constant $\sigma_J > 0$ such that $\|J_k^Tc_k\| \ge \sigma_J \|c_k\|$.
\end{ass}

The following lemma is an improved version of \cref{lemma: constraint model reduction}. Since LICQ holds for all iterations, we can now use the exact infeasibility measure $\varphi = \|c(x)\|$. 
\begin{lemma}\label{lemma: model reduction and cknorm}
    Let \mo{Assumptions \ref{assumption1}} and \ref{ass: LICQ} hold for all $k\in \bN$, then with $\|c_k\| > 0$, 
    \begin{equation}
         \|c_k\|-\|c_k+J_kv_k\| \ge \mo{\kappa_v\sigma_J \|J_k^T c_k\|} \ge \kappa_v\sigma_J^2 \|c_k\|.
    \end{equation}
\end{lemma} 
\begin{proof}
    From \cref{lemma: constraint model reduction}, 
    \begin{equation*}
        \|c_k\|(\|c_k\|-\|c_k+J_kv_k\|) \ge \kappa_v\|J_k^Tc_k\|^2 \ge \mo{\kappa_v\sigma_J \|J_k^T c_k\| \|c_k\|} \ge \kappa_v \sigma_J^2 \|c_k\|^2.
    \end{equation*}
    Dividing both sides by $\|c_k\|$ yields \mo{both results}.
\end{proof}

The following lemma shows that LICQ serves as a sufficient condition for the existence of $\taum > 0$, \mo{thus implying that Assumption \ref{ass:taumin} holds throughout this subsection}.

\begin{lemma}\label{lemma: taumin exists}
    Let Assumptions \ref{assumption1}, \ref{ass: assumption on Hk}, \ref{ass2} and \ref{ass: LICQ} hold and let $\sigma \in (0, 1)$ and $\beta_k \le \kappa_\beta$ hold for all $k\in \bN$. Then, for all $k \in \bN$, there exists such $\tau_{\min}:= \frac{(1-\sigma)\kappa_v\sigma_J}{\omega(\kappa_\beta\kappa_H \kappa_u + \kappa_g)}$ satisfying 
    \begin{equation*}
        \tau_{\min} (\nabla f(x_k)^T\dktrue + \beta_k (\uktrue)^TH_k \uktrue) \le (1-\sigma) (\|c_k\| - \|c_k+J_k\mo{v_k}\|).
    \end{equation*}
\end{lemma}

\begin{proof}
    By the definition of $\dktrue$,
    \begin{align*}
        \nabla f(x_k)^T\dktrue &= \nabla f(x_k)^T(\beta_k \uktrue + v_k) \\ 
        &= -\beta_k (\uktrue)^TH_k \uktrue - \beta_k v_k^TH_k\uktrue + \nabla f(x_k)^Tv_k.
    \end{align*}
    From the trust-region constraint in subproblem \cref{eq: subproblem v_k} \mo{and Lemma \ref{lemma: model reduction and cknorm}}, we have
    \begin{align*}
        \nabla f(x_k)^T  \dktrue + \beta_k (\uktrue)^TH_k \uktrue 
        &= - \beta_k v_k^TH_k\uktrue + \nabla f(x_k)^Tv_k \\
        &\le (\beta_k \kappa_H \|\uktrue\| + \|\nabla f(x_k)\|) \|v_k\| \\
        &\le (\kappa_\beta\kappa_H \kappa_u + \kappa_g) \omega \|J_k^Tc_k\| \\
        &\le \omega(\kappa_\beta\kappa_H \kappa_u + \kappa_g) \frac{(\|c_k\| - \|c_k+J_k\mo{v_k}\|)}{\kappa_v\sigma_J}.
    \end{align*}
    Multiplying $\tau_{\min}$ on both sides proves the result.
\end{proof}

    

We can now prove the main theorem for Case II.

\begin{theorem}\label{thm: LICQ hold complexity}
    Let Assumptions \ref{assumption1}, \ref{ass: assumption on Hk}, \ref{ass2} and \ref{ass: LICQ} hold for all $k\in \bN$. Choose $\beta_k = \eta/\sqrt{K}$ and $\ga_k$ as in \cref{alpha choice: interval}. Then for any $K\in \bN$,
    \begin{equation} \label{eq: average sum of Eck}
        \frac{1}{K}\sum_{k=0}^{K-1} \bE[\|c_k\|] \le \frac{2\kappa_2}{\nu \sigma\kappa_v \sigma^2_JK},
    \end{equation}
    and
    \begin{equation}
        \frac{1}{K}\sum_{k=0}^{K-1} \bE [\| \nabla f(x_k) + J_k^T\yktrue\|^2] \le \frac{2\kappa_H^2\kappa_2}{\nu \eta \taum \zeta\sqrt{K}} + \frac{4 \kappa_2 \kappa_H^2 \kappa_0 \kappa_J^2 \kappa_c}{\nu\sigma\kappa_v\sigma_J^2\zeta K}.    \end{equation}
\end{theorem}
\begin{proof}
    By \cref{lemma: taumin exists}, we know that $\taum > 0$ exists. Therefore, the conditions for \cref{thm: taum exists convergence} are satisfied and thus, by \eqref{eq: used in LICQ equation} and $\beta_k = \eta/\sqrt{K}$,
    \begin{equation*}
        \sum_{k=0}^{K-1} \bE [\alpha_k\beta_k\taum (\uktrue)^TH_k\uktrue + \frac{1}{2}\alpha_k\sigma(\|c_k\| - \|c_k+J_k\mo{v_k}\|)] \le \phi(x_0, \taum) - \taum f_{\inf}+ \kappa_1 \eta^2.
    \end{equation*}
    From \cref{lemma: model reduction and cknorm}, we have 
    \begin{equation*}
            \sum_{k=0}^{K-1} \bE [\alpha_k\beta_k\taum (\uktrue)^TH_k\uktrue + \frac12 \alpha_k \sigma \kappa_v \sigma_J^2\|c_k\|] \le \phi(x_0, \taum) - \taum f_{\inf}+ \kappa_1 \eta^2.
    \end{equation*}
    Therefore, by Assumption \ref{ass: assumption on Hk},
    \begin{equation*}
        \sum_{k=0}^{K-1} \frac{1}{2}\nu\sigma\kappa_v\sigma_J^2 \bE[\|c_k\|] \le \kappa_2.
    \end{equation*}
    Dividing both sides by $K$ yields the first result. Next, from \cref{lemma: uHutrue with KKT measure no LICQ}, we have
    \begin{align*}
        &\sum_{k=0}^{K-1} \bE \left[\frac{\zeta \alpha_k \beta_k \taum}{2\kappa_H^{2}} \|\nabla f(x_k) + J_k^Ty_k^\text{true}\|^2 - \alpha_k\beta_k\taum \kappa_0 \|J_k^T c_k\|^2\right] \\ \le &\sum_{k=0}^{K-1} \bE [\alpha_k\beta_k\taum (\uktrue)^TH_k\uktrue] \le \kappa_2.
    \end{align*}
    Therefore,
    \begin{equation}
        \sum_{k=0}^{K-1} \bE [\| \nabla f(x_k) + J_k^T\yktrue\|^2] \le \frac{2\kappa_H^2\kappa_2\sqrt{K}}{\nu \eta \taum \zeta} + 2 \zeta^{-1} \kappa_H^2 \kappa_0 \kappa_J^2 \kappa_c\sum_{k=0}^{K-1}  \bE[\|c_k\|].
    \end{equation}
    Dividing both sides by $K$ and using \cref{eq: average sum of Eck}, yields
    \begin{equation}\notag
        \frac{1}{K}\sum_{k=0}^{K-1} \bE [\| \nabla f(x_k) + J_k^T\yktrue\|^2] \le \frac{2\kappa_H^2\kappa_2}{\nu \eta \taum \zeta\sqrt{K}} + \frac{4 \kappa_2 \kappa_H^2 \kappa_0 \kappa_J^2 \kappa_c}{\nu\sigma\kappa_v\sigma_J^2\zeta K}. 
    \end{equation}
\end{proof}

Parallel to Corollary \ref{cor: Jkck2 complexity} and \ref{cor: reduced gradient complexity}, we can obtain worst-case complexity results from \at{\cref{thm: LICQ hold complexity}}. To achieve $\bE[\|c_k\|]\le \epsilon_c$, we only need at most $\mathcal{O}(\epsilon_c^{-1})$ iterations. This result matches the complexity of the algorithm proposed in \cite{o2024two} to reach an $\epsilon_c$-feasible point (although it uses a $l_1$-norm and assumes the exact $d_k$ is accessible), which is \mo{the best known result when LICQ holds}. Clearly, the improvement \mo{when compared with the results of Section \ref{sec: when taum and no licq} is directly as a result of Assumption \ref{ass: LICQ}. On the other hand, we still require at most $\mathcal{O}(\epsilon_L^{-4})$ iterations to achieve $\bE[\|\nabla f(x_k)+J_k^T\yktrue\|]\le \epsilon_L$, matching the optimal complexity with respect to this measure.}

\subsection{Case III: Merit parameter sequence vanishes}

For the analysis in the previous subsections, we assumed that $\taum > 0$ exists, even though our algorithm does not directly compute it. Unfortunately, in certain cases, such $\taum$ does not exist, \ie the merit parameter sequence $\{\tau_k\}$ may decrease to 0 in the limit.  \mo{This behavior is fundamentally tied to the existence of a (sub)sequence of iterates along which the constraint} Jacobians tends toward rank deficiency. \mo{In such a case, the best result we can hope for is convergence to an infeasible stationary point, i.e., convergence in $\|J_k^T c_k\|$. In this setting, we can prove our convergence result under the less restrictive choice of $\nu$, given by,
\begin{equation} \label{eq:nuvanishingtau}
    \nu \in \left(0,  \min\left\{ \frac{\kappa_v\kappa_c^{-1}}{2\Gamma\omega^2}, 1-\theta \kappa_\beta\right\}\right].
\end{equation}
}

Armed with this definition, we now derive the fundamental lemma of this section.

\begin{lemma} \label{lemma:c_k-c_k+1new}
    Under Assumptions \ref{assumption1}, \ref{ass: assumption on Hk} and \ref{ass2}, it follows for all $k\in \bN$ that
    \mo{\begin{equation}\label{eq: tauk vanish summing new}
        \bE_k [\|c_k\| - \|c (x_k+\alpha_k d_k)\|] \ge \frac{1}{2}\nu\kappa_v\kappa_c^{-1} \bE_k[\|J_k^Tc_k\|^2] - \kappa_3 \beta_k^2,
    \end{equation}
    where we define $\kappa_3 := \Gamma( \theta^2\kappa_J^2 \kappa_c^2 \omega^2 + \frac{1}{2} (\zeta^{-1}\kappa_H\kappa_u^2 + \zeta^{-2}M)).$}
\end{lemma}

\begin{proof}
    By Lipschitz continuity of the Jacobian of $c$ and $\alpha_k \leq 1$,
    \begin{align*}
        \|c_k\| - \|c(x_k+ \alpha_k d_k)\| 
        & \ge \alpha _k (\|c_k\| - \|c_k + J_k v_k\|) - \frac{\Gamma}{2} \alpha_k^2 \|d_k\|^2  \\ 
        & \ge \nu \kappa_v \kappa_c^{-1} \|J_k^T c_k\|^2 - \frac{\Gamma}{2} \alpha_k^2 (\|v_k\|^2 + \beta_k^2 \|u_k\|^2)  \\ 
        & \ge \nu \kappa_v \kappa_c^{-1} \|J_k^T c_k\|^2 - \Gamma (\nu^2 + \theta^2 \beta^2) \|v_k\|^2 - \frac{\Gamma\beta_k^2}{2} \|u_k\|^2  \\ 
        & \ge \nu \kappa_v \kappa_c^{-1} \|J_k^T c_k\|^2 - \Gamma \nu^2 \omega^2 \|J_k^T c_k\|^2 - \beta_k^2 \Gamma \theta^2   \kappa_J^2 \kappa_c^2 \omega^2  - \frac{\Gamma\beta_k^2}{2} \|u_k\|^2  \\ 
        & \ge \frac12 \nu \kappa_v \kappa_c^{-1} \|J_k^T c_k\|^2 - \beta_k^2 \Gamma \theta^2   \kappa_J^2 \kappa_c^2 \omega^2  - \frac{\Gamma\beta_k^2}{2} \|u_k\|^2
    \end{align*}
    Taking the conditional expectation on both sides and applying Lemmas \ref{lemma2: u_k norm} and \ref{lem: upper bound of norm u_k} yields the result.
    \end{proof}

Now we present our main result for this section.

\begin{theorem}\label{thm: taum not exist new}
    Let Assumptions \ref{assumption1}, \ref{ass: assumption on Hk} and \ref{ass2} hold for all $k\in \bN$. Set $\beta_k = \eta/\sqrt{K}$, $\ga_k$ as in \cref{alpha choice: interval}, and $\nu$ as in \eqref{eq:nuvanishingtau}. Then for any $K\in \bN$,
    \begin{equation} \label{eq: J_kc_k no taumin}
        \frac{1}{K} \sum_{k=0}^{K-1} \bE[\|J_k^Tc_k\|^2] \le \frac{2 \kappa_c \left(\|c_0\| + \eta^2 \kappa_3\right)}{\nu \kappa_v K}.
    \end{equation}
\end{theorem}

\begin{proof}
    Summing \eqref{eq: tauk vanish summing new} for $k=0,\dots,K-1$ and rearranging, we have
    \begin{equation*}
        \quad \ \frac12 \nu \kappa_v \kappa_c^{-1} \sum_{k=0}^{K-1} \bE_k[\|J_k^Tc_k\|^2]
        \le \sum_{k=0}^{K-1} \bE_k[\|c_k\| - \|c (x_{k+1})\|] + \kappa_3 \beta_k^2.
    \end{equation*}
    Taking the total expectation and applying $\beta_k = \eta/\sqrt{K}$, we have
    \begin{equation*}
        \frac12 \nu \kappa_v \kappa_c^{-1} \sum_{k=0}^{K-1} \bE[\|J_k^Tc_k\|^2] \le \|c_0\| + \eta^2 \kappa_3.
    \end{equation*}
    Rearranging terms and dividing through by $K$ yields the result.
\end{proof}

To our knowledge, this represents the first complexity result for convergence to an infeasible stationary point, and matches the $\mathcal{O}(\epsilon_c^{-2})$ complexity result obtained in \Cref{sec: when taum and no licq}. While convergence to an infeasible stationary point is obviously undesired, it is the best possible convergence behavior one can hope for under such mild assumptions, even when $f$ is deterministic.

\section{The Inexact Version of Stochastic SQP Algorithm}\label{sec: inexact algs}

In our previous discussion, we assumed that an \emph{exact} solution of the subproblem \cref{eq: subproblem u_k} is obtained so that condition $u_k \in \text{Null}(J_k)$, \ie \  $J_ku_k = 0$ is strictly satisfied. However, such an exact $u_k$ is expensive to compute in practice, especially when the problem is large-scale. More importantly, since $u_k$ \mo{arises from} the stochastic subproblem \cref{eq: subproblem u_k} involving $g_k$, it is natural that we need not expend too much computation solving for $u_k$, which is fundamentally corrupted by stochastic noise. Therefore, we turn to \emph{inexact} solvers for the subproblem \cref{eq: subproblem u_k}. We introduce termination tests, which are used as a stopping criterion for iterative linear system solvers such as MINRES and enable our analysis to track the residuals of the linear system. We now introduce and (re)define some variables for the analysis in this section. 

\mo{In this section, we assume $u_k$ and $y_k$ are computed} using an inexact iterative linear system solver. \mo{This solver also generates a pair of residuals} $(\rho_k, r_k)\in \bR^n \times\bR^m$ such that, for every iteration $k\in \bN$, $(\rho_k, r_k)$ satisfy our termination tests. \mo{We denote the exact solution of the subproblem \cref{eq: subproblem u_k} by $u_{k,*}$}, which lies in the Null space of $J_k$ (\ie $u_{k,*}$ is the $u_k$ used in previous sections). Recalling the general KKT system \cref{eq: subproblem u_k's KKT system} of the tangential subproblem, we define the residual pair $(\rho_k, r_k)$ formally as follows:
\begin{equation}\label{eq: define rho_k and r_k}
    \begin{bmatrix}
        \rho_k \\ r_k
    \end{bmatrix} :=\begin{bmatrix}
        H_k & J_k^T \\J_k & 0 
    \end{bmatrix}\begin{bmatrix}
        u_k \\y_k
    \end{bmatrix} + \begin{bmatrix}
        g_k + H_kv_k \\0
    \end{bmatrix}.
\end{equation}
We make the following assumptions on the iterative solver we use. Note that Assumption \ref{ass: iterative solver} is a common property of linear system solvers. In fact, Krylov Subspace solvers, such as MINRES, can produce an exact solution \ie \ $(\rho_k, r_k) = (0,0)$ after \mo{$t=n+m$} iterations.
\begin{ass}\label{ass: iterative solver}
    For all $k\in\bN$, the iterative system solver used to compute $u_k$ and $y_k$ generates a sequence $\{(u_{k,t}, \rho_{k, t}, r_{k, t})\}_{t\ge 0}$ with 
    \begin{equation}
    \begin{bmatrix}
        \rho_{k, t} \\ r_{k, t}
    \end{bmatrix} :=\begin{bmatrix}
        H_k & J_k^T \\J_k & 0 
    \end{bmatrix}\begin{bmatrix}
        u_{k, t} \\y_{k, t}
    \end{bmatrix} + \begin{bmatrix}
        g_k + H_kv_k \\0
    \end{bmatrix},
\end{equation}
    such that $\lim_{t\to \infty} \|(\ukt, \rho_{k,t}, r_{k, t})-(u_{k,*},  0, 0)\| =0$.
\end{ass}

Note that, since the $y_k$ part of the solution to \eqref{eq: subproblem u_k} is not unique when $J_k$ is rank deficient and we do not need $y_k$ in our analysis, we do not include it in the assumption. Formally, we ask for the following termination tests of the iterative linear system solver.
\begin{ttt}\label{tt: termination test on residual norms}
    Let $(\gamma_r, \gamma_\rho) \in \bR_{>0} \times \bR_{>0}$ and $\beta_k \in \bR_{>0}$ be a predefined rescaling parameter. For all $k\in\bN$, we terminate the iterative linear solver as long as the residuals satisfy the following conditions:
    $$
    \|r_k\| \le \gamma_r\beta_k \quad \text{and} \quad \|\rho_k\| \le \gamma_\rho\beta_k.
    $$
\end{ttt}
Under Assumption \ref{ass: iterative solver}, the termination test is \mo{always satisfied for all sufficiently large $t$.}

Since our inexact $u_k$ \mo{is no longer contained in the} Null space of $J_k$ now, we \mo{first prove a few} parallel lemmas to \cref{lem: upper bound of norm u_k} to continue our convergence analysis.


We now derive the explicit form of our inexact $u_k$. With $J_k^+$, we can write $u_k$ as the decomposition $u_k = u_{k,1} + u_{k,2}$, where $u_{k,1} \in \text{Range}(J_k)$ and $u_{k,2}\in \text{Null} (J_k)$.

From \cref{eq: define rho_k and r_k}, we have 
\begin{equation}\label{eq: uk1 and uk2}
    \begin{gathered}
        u_{k,1} = J_k^+ r_k, \quad 
        u_{k,2} = -Z_k(Z_k^TH_kZ_k)^{-1}Z_k^T (g_k+H_kv_k + H_kJ_k^+r_k - \rho_k).
    \end{gathered}
\end{equation}
Thus the \emph{inexact} solution $u_k$ of \eqref{eq: define rho_k and r_k} is 
\begin{equation}\label{eq: inexact u_k explicit form}
    u_k = J_k^+ r_k-Z_k(Z_k^TH_kZ_k)^{-1}Z_k^T (g_k+H_kv_k + H_kJ_k^+r_k - \rho_k).
\end{equation}
Also, recall that the tangential component computed \mo{exactly} with the true gradient is
$$
\uktrue = -Z_k(Z_k^TH_kZ_k)^{-1}Z_k^T(\nabla f(x_k)+ H_kv_k).
$$
To derive a uniform bound on $\|u_k\|^2$, we make the following assumption on the behavior of constraint Jacobian matrices $\{J_k\}$.
\begin{ass}\label{ass: uniform lower singular value}
    For all $k\in\bN$, the nonzero singular values of $J_k$ is uniformly lower bounded by some positive value $\bar\sigma_{\min} \in \bR_{>0}$.
\end{ass}
This assumption covers the LICQ case, where all the singular values of the constraint Jacobians are positive and bounded from below. To clarify, since the assumption does not require all singular values to be positive, it includes certain ``ideal" cases of rank-deficient Jacobians. \mo{We note that this requirement is necessary to provide any meaningful bound on $\|u_{k,1}\|$ and thus is fundamental to our analysis with an inexact $u_k$.} In particular, Assumption \ref{ass: uniform lower singular value} yields a uniform upper bound that is $\|J_k^+\|\le 1/\bar\sigma_{\min}$ for all $k\in\bN$.
\begin{lem}\label{lem: bound on norm of expecation u_k-uktre}
    \mo{Let Assumptions \ref{ass: iterative solver} and \ref{ass: uniform lower singular value} hold.} Then, for all $k\in\bN, \|\bE_k[u_k - \uktrue]\| \le \kappa_4 \beta_k$, where $\kappa_4 := (\bar\sigma_{\min}^{-1}\gamma_r + \zeta^{-1} \gamma_\rho)$.
\end{lem}

\begin{proof}
    \at{From \cref{eq: inexact u_k explicit form} we know}
    \begin{align*}
        \bE_k [u_k - \uktrue] &= \bE_k [J_k^+r_k - Z_k(Z_k^TH_kZ_k)^{-1} Z_k^T(g_k - \nabla f(x_k)) \\ 
        &\qquad - Z_k(Z_k^TH_kZ_k)^{-1} Z_k^TH_kJ_k^+r_k + Z_k(Z_k^TH_kZ_k)^{-1} Z_k^T \rho_k] \\ 
        &= \bE_k [(I- Z_k(Z_k^TH_kZ_k)^{-1}Z_k^TH_k) J_k^+ r_k] + \bE_k[Z_k(Z_k^TH_kZ_k)^{-1}Z_k^T\rho_k].
    \end{align*}
    From \at{\cite[Lemma 11]{berahas2024stochastic}}, we know $\|I- Z_k(Z_k^TH_kZ_k)^{-1}Z_k^TH_k\|\le 1$ and $\|Z_k(Z_k^TH_kZ_k)^{-1}Z_k^T\|\le \zeta^{-1}$. 
    Taking the norm of both sides and using Cauchy-Schwarz and the triangle inequality, we have 
    \begin{align*}
        \|\bE_k(u_k - \uktrue)\| & \le \|\bE_k [(I- Z_k(Z_k^TH_kZ_k)^{-1}Z_k^TH_k) J_k^+ r_k]\| \\ 
        & \qquad + \|\bE_k[Z_k(Z_k^TH_kZ_k)^{-1}Z_k^T\rho_k]\| \\
        & \le \frac{1}{\bar \sigma_{\min}}\gamma_r \beta_k + \zeta^{-1}\gamma_\rho \beta_k,
    \end{align*}
    which proves the result.
\end{proof}

\begin{lem}\label{lem: upper bound on inexact u_k norm}
    \mo{Let Assumptions \ref{assumption1}, \ref{ass: assumption on Hk}, \ref{ass2}, \ref{ass: iterative solver}, and \ref{ass: uniform lower singular value} hold.} For all $k\in\bN$, $\bE_k\|u_k\|^2$ is bounded from above by $\tilde\kappa_u:= 2\kappa_u^2 + 2 (\zeta^{-1}\sqrt{M} + \kappa_4 \mo{\kappa_\beta})^2$.
\end{lem}
\begin{proof}
    From \cite[Lemma 11]{berahas2024stochastic}, we know that $\bE_k[\|u_k - \uktrue\|] \le \zeta^{-1}\sqrt{M} + \kappa_4 \beta_k$. \\ 
    Therefore, $\bE_k \|u_k\|^2 \le \bE_k [2(\|\uktrue\|^2 + \|u_k - \uktrue\|^2)] \le 2\kappa_u^2 + 2 (\zeta^{-1}\sqrt{M} + \kappa_4 \mo{\kappa_\beta})^2 = \tilde\kappa_u$.
\end{proof}

\mo{Now, we present our first main result of this section, when Assumption \ref{ass:taumin} holds as well.}

\begin{thm}\label{thm: inexact main theorem}
    Let Assumptions \ref{assumption1}, \ref{ass: assumption on Hk}, \ref{ass2}, \mo{\ref{ass:taumin}}, \ref{ass: iterative solver}, and \ref{ass: uniform lower singular value} hold and let $\beta_k = \eta/\sqrt{K}$ and $\ga_k\in[\nu, \nu+\theta\beta_k]$. Let
    \mo{
    \begin{equation*}
        \nu \in \left(0, \min\left\{ \frac{\sigma\kappa_v\kappa_c^{-1}}{4(\taum L+\Gamma)\omega^2}, 1-\theta \kappa_\beta\right\}\right),
    \end{equation*}}
    and define 
    \begin{equation*}
        \tilde\kappa_1 := \taum\kappa_g\kappa_4 + (\taum L+\Gamma)\tilde \kappa_u + \mo{2}\theta^2(\taum L+\Gamma)\go^2\kappa_J^2\kappa_c^2 + \gamma_r,
    \end{equation*}
    \begin{equation*} \text{and} \quad
        \tilde \kappa_2 := \taum(\mo{f_0}-f_{\inf})+\|c_0\|+\eta^2\tilde\kappa_1.
    \end{equation*}
    Then, we have
    \begin{equation}
        \sum_{k=0}^{K-1} \bE[\nu\beta_k\taum (\uktrue)^TH_k\uktrue + \frac{1}{2}\nu\sigma (\|c_k\|-\|c_k+J_kv_k\|)] \le \tilde\kappa_2.
    \end{equation}
\end{thm}

\begin{proof}
Similar to the proof of \cref{lemma: generic descent},
\begin{align*}
    & \quad \ \phi (x_k +\alpha_k d_k, \taum) - \phi (x_k ,\taum ) \\
    & \le -\ga_k \Delta l(x_k, \taum,\nabla f_k, \dktrue) + \ga_k \taum \nabla f_k^T(d_k - \dktrue) \\
    & \qquad- \ga_k (\|c_k + J_kv_k\| - \|c_k +J_kv_k+\beta_k r_k\|) + \frac{\mo{\alpha_k^2}}{2}(\taum L+\Gamma) \mo{\|d_k\|^2}.
\end{align*}
Using $\|d_k\|^2 \le 2(\gb_k^2\|u_k\|^2 + \|v_k\|^2)$ and rearranging the order of the inequality, we have 
\begin{align*}
    &\quad \ \ga_k \Delta l(x_k ,\taum \nabla f_k, \dktrue) \\ 
    & \le (\phi(x_k,\taum)- \phi(x_k+\ga_kd_k,\taum)) + \alpha_k\taum \nabla f_k^T(d_k - \dktrue) \\
    & \qquad+ \ga_k (\|c_k +J_kv_k+\beta_k r_k\| - \|c_k + J_kv_k\|) + (\taum L+\Gamma) \ga_k^2 (\beta_k^2\|u_k\|^2 + \|v_k\|^2).
\end{align*}
\mo{By the definition of $\nu$ and $\ga_k$}, taking the conditional expectation, by \eqref{eq: local model reduction},
\begin{align*}
    &\quad \ \nu\bE_k [\taum \beta_k (\uktrue)^T H_k\uktrue + \sigma (\|c_k\| - \|c_k+J_kv_k\|)] \\
    & \le \bE_k [(\phi(x_k,\taum) - \phi (x_k+\alpha_kd_k, \taum)] 
        +\bE_k [\alpha_k\taum\nabla f_k^T(d_k - \dktrue)] \\
    & \qquad +\alpha_k \beta_k \bE_k[\|r_k\|] + \bE_k [\alpha_k^2\beta_k^2(\taum L+\Gamma) \|u_k\|^2] \\
    & \qquad + \mo{2}\nu^2(\taum L+\Gamma)\omega^2 \|J_k^Tc_k\|^2 + \mo{2}\theta^2\gb_k^2(\taum L+\Gamma)\omega^2\kappa_J^2\kappa_c^2.
\end{align*}
\mo{Working with the second term on the right-hand side of this inequality, by Lemma \ref{lem: upper bound on inexact u_k norm},
\begin{equation} \label{eq:inexactgdk}
    \bE_k [\alpha_k\taum\nabla f_k^T(d_k - \dktrue)] = \bE_k [\alpha_k\taum \beta_k \nabla f_k^T(u_k - \uktrue)] \le \beta_k^2 \taum \kappa_g \kappa_4.
\end{equation}}
Since $\ga_k\le 1$ and $\|r_k\| \le \gamma_r\beta_k$, we know 
$$\alpha_k\beta_k \bE_k \|r_k\| \le \gamma_r\beta_k^2.$$
By the choice of $\nu$ and \cref{lemma: constraint model reduction}  we have
\begin{equation*}
    \mo{2}\nu^2(\taum L+\Gamma) \omega^2\|J_k^Tc_k\|^2 \le \frac{1}{2}\nu \sigma (\|c_k\|-\|c_k+J_kv_k\|).
\end{equation*}
Combining the results in Lemma \ref{lem: bound on norm of expecation u_k-uktre} and \ref{lem: upper bound on inexact u_k norm}, we can conclude that 
\begin{align*}
    & \quad \ \bE_k [\nu\taum \beta_k (\uktrue)^T H_k \uktrue +\frac{1}{2} \nu\sigma(\|c_k\|-\|c_k+J_kv_k\|)] \\ &\le \bE_k[\phi(x_k,\taum) - \phi(x_k+\ga_kd_k, \taum)] \\ 
    & \qquad + \beta_k^2 (\taum\kappa_g \kappa_4 + (\taum L+\Gamma)\tilde \kappa_u + \mo{2}\theta^2 (\taum L+\Gamma)\go^2\kappa_J^2\kappa_c^2 + \gamma_r) \\  &= \bE_k[\phi(x_k,\taum) - \phi(x_k+\ga_kd_k, \taum)] + \tilde\kappa_1\beta_k^2.
\end{align*}
Thus, summing from $k=0$ to $K-1$, and substituting $\beta_k = \eta / \sqrt{K}$, we have 
\begin{equation}
    \sum_{k=0}^{K-1} \bE[\nu\beta_k\taum (\uktrue)^TH_k\uktrue + \frac{1}{2}\nu\sigma (\|c_k\|-\|c_k+J_kv_k\|)] \le \tilde\kappa_2.
\end{equation}
\end{proof}

Theorem \ref{thm: inexact main theorem} \mo{is an essentially parallel result to} Theorem \ref{thm: taum exists convergence}. The difference lies mainly in the constants, \mo{as} $\tilde\kappa_1$ includes additional terms that \mo{are dependent on} $\gamma_\rho$ and $\gamma_r$. This is natural in that our Termination Test \ref{tt: termination test on residual norms} forces the errors of residuals below a threshold that is controlled by $\beta_k$. 

For the sake of brevity, we only list the complexity results of the inexact algorithm under two different cases. Their proofs follow \mo{essentially in the same manner as those of} Corollary \ref{cor: Jkck2 complexity}, \ref{cor: reduced gradient complexity}, and Theorem \ref{thm: LICQ hold complexity}.

\begin{cor}
    Under the same assumptions as in Theorem \ref{thm: inexact main theorem}, we have the following results:
    \begin{enumerate}
        \item When no constraint qualification is assumed, we have 
        \begin{gather*}
            \frac{1}{K}\sum_{k=0}^{K-1} \bE [\|J_k^T c_k\|^2] \le \frac{2\kappa_c\tilde\kappa_2}{\nu \sigma \kappa_v K}, \,\, \text{and} \\\quad \frac{1}{K}\sum_{k=0}^{K-1} \mo{\bE[\|\nabla f(x_k)+J_k^T \yktrue\|^2]} \le \ \frac{\mo{2}\kappa_H^2\tilde\kappa_2}{\nu\eta\taum\zeta\sqrt{K}} + \frac{4\kappa_H^2 \kappa_c\kappa_0\tilde\kappa_2}{\nu \sigma\kappa_v \zeta K}.
        \end{gather*}
        \item When LICQ holds for all iterate $k$, we have
        \begin{gather*}
            \frac{1}{K}\sum_{k=0}^{K-1} \bE [\|c_k\|] \le \frac{2\tilde\kappa_2}{\nu \sigma\kappa_v\sigma_J^2 K}, \,\, \text{and} \\  \frac{1}{K}\sum_{k=0}^{K-1} \bE [\| \nabla f(x_k) + J_k^T\yktrue\|^2] \le \frac{\mo{2}\kappa_H^2\tilde\kappa_2}{\nu \eta \taum \zeta\sqrt{K}} + \mo{\frac{4 \tilde\kappa_2 \kappa_H^2 \kappa_0 \kappa_J^2 \kappa_c}{\nu\sigma\kappa_v\sigma_J^2\zeta K}}. 
        \end{gather*}        
   \end{enumerate}
\end{cor}

Again, we need at most $\mathcal{O}(\epsilon_c^{-2})$ and $\mathcal{O}(\epsilon_c^{-1})$ iterations to reach a $\epsilon_c$-stationary feasible point for case I and case II, respectively. The only difference lies in the constant. For example, the inexact algorithm needs roughly $(\tilde\kappa_2 - \kappa_2)\times 2(\kappa_c/\nu\sigma\kappa_v)$ more iterations in the worst case. \mo{With respect to the gradient of the Lagrangian}, both cases yield the same $\mathcal{O}(\epsilon_L^{-4})$ complexity.

\mo{Next, we turn our attention to the case where $\taum$ does not exist and provide a lemma parallel to \cref{lemma:c_k-c_k+1new}.}

\begin{lemma} \label{lemma:c_k-c_k+1inexact}
    Let $\nu$ be defined by \eqref{eq:nuvanishingtau}. Then, under Assumptions \ref{assumption1}, \ref{ass: assumption on Hk}, \ref{ass2}, \ref{ass: iterative solver}, and \ref{ass: uniform lower singular value}, it follows for all $k\in \bN$ that
    \begin{equation}\label{eq: inexact tauk vanish summing new}
        \bE_k [\|c_k\| - \|c (x_k+\alpha_k d_k)\|] \ge \frac{1}{2}\nu\kappa_v\kappa_c^{-1} \bE_k \|J_k^Tc_k\|^2 - \tilde\kappa_3 \beta_k^2,
    \end{equation}
    where we define 
    $$
    \tilde\kappa_3 := \Big[\Gamma (\frac{1}{2}\tilde\kappa_u + \theta^2\kappa_J^2 \kappa_c^2 \omega^2) + \gamma_r\Big].
    $$
\end{lemma}
\begin{proof}
    The proof follows directly by the proof of \cref{lemma:c_k-c_k+1new} combined with the techniques used in the proof of \cref{thm: inexact main theorem}.
\end{proof}

Given this, we can present our final complexity result.

\begin{theorem}\label{thm: taum not exist inexact}
    Let Assumptions \ref{assumption1}, \ref{ass: assumption on Hk}, \ref{ass2}, \ref{ass: iterative solver}, and \ref{ass: uniform lower singular value} hold for all $k\in \bN$. Set $\beta_k = \eta/\sqrt{K}$, $\ga_k$ as in \cref{alpha choice: interval}, and $\nu$ as in \eqref{eq:nuvanishingtau}. Then for any $K\in \bN$,
    \begin{equation} \label{eq: J_kc_k no taumin, inexact}
        \frac{1}{K} \sum_{k=0}^{K-1} \bE[\|J_k^Tc_k\|^2] \le \frac{2 \kappa_c \left(\|c_0\| + \eta^2 \tilde\kappa_3\right)}{\nu \kappa_v K}.
    \end{equation}
\end{theorem}

\begin{proof}
    The proof follows via an identical argument as the proof of \cref{thm: taum not exist new}.
\end{proof}

Clearly, this result parallels that of \cref{thm: taum not exist new} and proves a worst-case complexity of $\mathcal{O}(\epsilon_c^{-2})$ with respect to our infeasibility measure. Thus, all of our complexity results translate directly to the setting of inexact solves for $u_k$, under Assumptions \ref{ass: iterative solver} and \ref{ass: uniform lower singular value}.


\section{Numerical Experiments}\label{sec: numerical experiments}

In this section, we validate the performance of our inexact two-stepsize SQP (ITSQP) method with numerical experiments. We tested ITSQP together with the original stochastic SQP method (SSQP) \cite{berahas2021sequential} and a step-lengthening stochastic SQP method (SSQPL) on a subset of the equality constrained optimization problems from the CUTEst collection \cite{gratton2025s2mpj}. Since the singularity of Jacobians is detected in more than half of the problems, we use the step decomposition strategy for all three algorithms to make the results comparable.

We use an experiment setup similar to \cite{berahas2021sequential}, with a total of 110 equality constrained problems. Multiple levels of stochasticity (noise that follows multivariate normal distribution $\cN(0, \epsilon_NI)$, where $\epsilon_N\in\{10^{-4}, 10^{-3}, 10^{-2}, 10^{-1}\}$) are injected into the gradient estimates to simulate the stochastic gradient. To ``tune" the algorithms, we tested 5 different fixed scaling stepsizes $\beta \in \{10^{-4}, 10^{-3}, 10^{-2}, 10^{-1},1\}$ under 15 different seeds. Therefore, a total of 440(=110$\times$4) problems are tested, and for each problem, every algorithm is run 75(=$5\times 15$) times with different stepsizes. For the stepsize $\alpha_k$, we use the default stepsize \mo{selection} scheme in \cite{berahas2021sequential} for SSQP, a step-lengthening stepsize for SSQPL \at{and the adaptive $\ga_k$ proposed in \cite[Algorithm 4.1]{o2024two} for ITSQP}.

For each run of the algorithm, we give a budget of 10,000 iterations and report the best iterate with the following scheme: for any iterate $k$, if $\|c_k\|_\infty \le 10^{-6}$, we treat it as a feasible point. We then pick the best iterate with the lowest KKT error $\|\nabla f_k + J_k^T\yktrue\|_\infty$ from all feasible points. If $\|\nabla f_k + J_k^T\yktrue\|_\infty\le 10^{-4}$ \mo{on an iterate satisfying $\|c_k\|_\infty \le 10^{-6}$}, we terminate the algorithm. If such feasible points do not exist, we report the first-order measure of the most feasible point (one with the lowest $\|c_k\|_\infty$). This is commonly used in the stochastic SQP literature, such as \cite{curtis2024stochastic} and \cite{o2024two}.

The figures contain the KKT errors and infeasibility errors of three algorithms tested under different noise levels. For each of the 440 problems, we compared the average infeasibility errors and KKT errors over five different stepsizes and chose the one stepsize that yields the best results to plot the figures. It is \at{$\beta_k = 1$} for SSQP and SSQPL, and $\beta_k = 10^{-3}$ for ITSQP. 

We can conclude from the figures that our ITSQP method outperforms the SSQP and SSQPL methods in terms of the infeasibility measure across all noise levels. This advantage becomes increasingly pronounced as the noise level increases, highlighting the superior convergence rate of our two-stepsize algorithm.

The SSQP and SSQPL methods compute the merit parameter sequence $\tau_k$ and choose the stepsize based on $\tau_k$, which appears to help them achieve better KKT convergence. The higher first-order errors may also be attributed to the more pessimistic stepsize $\alpha_k$ we use in our algorithm.

\begin{figure}[htbp]
    \centering
    \includegraphics[width=\textwidth]{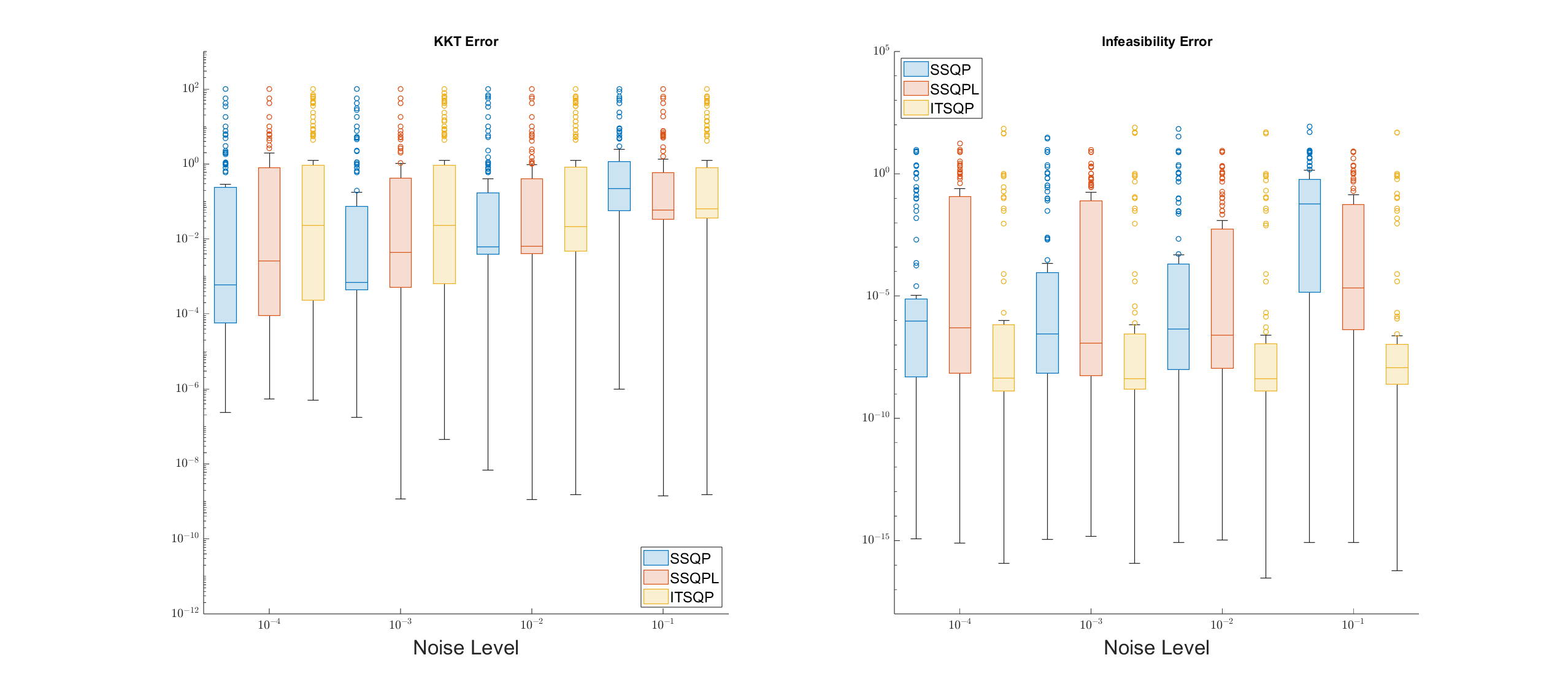}
    \caption{Comparison of KKT error and infeasibility error under different noise levels}
    \label{fig: errors comparison}
\end{figure}

\section{Conclusion}\label{sec: conclusion}

In this paper, we proposed and analyzed an inexact two-stepsize stochastic SQP method that can handle the possibility of rank deficient Jacobians. We prove the first-known $\mathcal{O}(\epsilon_c^{-2})$ complexity for the infeasibility measure $\|J_k^Tc_k\|$ to fall below $\epsilon_c$ under mild assumptions, and an improved $\mathcal{O}(\epsilon_c^{-1})$ complexity when LICQ holds, which matches the best known result in the literature. These results also hold in the case where the tangential component is computed inexactly. Numerical experiments also show that our algorithm converges more efficiently in terms of the infeasibility measure than those without the two-stepsize scheme.

Extending similar SQP strategies to the more general, inequality constrained setting, remains an open problem. In addition, extending this method to the case of stochastic equality constraints is a natural future direction.



\bibliographystyle{siamplain}
\bibliography{references}
\end{document}